\documentclass[12pt,reqno]{amsart}
\usepackage{graphicx}
\usepackage{amssymb}
\usepackage{amsfonts}
\usepackage{amsmath,latexsym,amssymb,amsfonts,amsbsy, amsthm, enumerate, enumitem}
\usepackage[usenames]{color}
\usepackage{xspace,colortbl}
\usepackage{mathrsfs}
\usepackage[colorlinks,linkcolor=blue,citecolor=blue]{hyperref}
\usepackage[titletoc]{appendix}
\usepackage{pdfsync}
\usepackage{esint}
\allowdisplaybreaks

\newcommand{\beq}{\begin{equation}}
	\newcommand{\eeq}{\end{equation}}
\newcommand{\ben}{\begin{eqnarray}}
	\newcommand{\een}{\end{eqnarray}}
\newcommand{\beno}{\begin{eqnarray*}}
	\newcommand{\eeno}{\end{eqnarray*}}
\newcommand{\R}{\mathbb{R}}
\newtheorem{thm}{Theorem}[section]
\newtheorem{defi}[thm]{Definition}
\newtheorem{lem}[thm]{Lemma}
\newtheorem{prop}[thm]{Proposition}
\newtheorem{coro}[thm]{Corollary}
\newtheorem{rmk}[thm]{Remark}

\def\XXint#1#2#3{{\setbox0=\hbox{$#1{#2#3}{\int}$ }
		\vcenter{\hbox{$#2#3$ }}\kern-.6\wd0}}

\title[Liouville  for supercritical Fujita equation]{A Liouville theorem for supercritical Fujita equation and its applications}

\author[K. Wang]{Kelei Wang$^\dag$}
\address{$^\dag$School of Mathematics and Statistics \\ Wuhan University\\
	Wuhan 430072, China}
\email{wangkelei@whu.edu.cn}

\author[J. Wei]{Juncheng Wei$^\ast$}
\address{$^\ast$Department of Mathematics \\ Chinese University of Hong Kong\\
	Shatin, NT, Hong Kong}
\email{wei@math.cuhk.edu.hk}

\author[K. Wu]{Ke Wu$^\ddag$}
\address{$^\ddag$Department of Mathematics \\ Chinese University of Hong Kong\\
	Shatin, NT, Hong Kong}
\address{$^\ddag$School of Mathematics\\ Yunnan Normal University\\ Kunming, 650500, China}
\email{kewu001@cuhk.edu.hk}

\thanks{K. Wang and J. Wei are both partially supported by  National Key R\&D Program of China (No. 2022YFA1005602). K. Wang is also supported by  the National Natural Science Foundation of China (No. 12131017, No. 12221001 and No. 12425108).  J. Wei is also  supported  by  Hong Kong General Research Fund ``New frontiers in singular limits of nonlinear partial differential equations". K. Wu is supported by the  National Natural Science Foundation
of China (No. 12401264).  We are grateful to the referee for constructive suggestions, which have been invaluable in enhancing the quality of this paper.}



\begin{document}

\begin{abstract}
We prove a Liouville theorem for ancient solutions to the supercritical Fujita equation
\[\partial_tu-\Delta u=|u|^{p-1}u, \quad -\infty <t<0, \quad p>\frac{n+2}{n-2},\]
which says if $u$ is close to the ODE solution $u_0(t):=(p-1)^{-\frac{1}{p-1}}(-t)^{-\frac{1}{p-1}}$ at large scales, then it is an ODE solution (i.e. it depends only on $t$). This implies a stability property for ODE blow-ups in this problem.

 As an application of these results, we show that for  a suitable weak solution, its singular set  at the end time can be decomposed into two parts: one part is relatively open and $(n-1)$-rectifiable, and it is characterized by the property that tangent functions at these points are   the two constants $\pm(p-1)^{-\frac{1}{p-1}}$;  the other part is relatively closed and its Hausdorff dimension is not larger than $n-\left[2\frac{p+1}{p-1}\right]-1$.
\end{abstract}

\maketitle

\tableofcontents


\section{Introduction}\label{sec introduction}
	\setcounter{equation}{0}

This paper is a continuation of our previous work \cite{Wang-Wei-Wu} on the blow-up analysis for Fujita equation (or nonlinear heat equation)  
\begin{equation}\label{eqn}\tag{F}
    \partial_tu-\Delta u=|u|^{p-1}u
\end{equation}
in the supercritical case $p>\frac{n+2}{n-2}$ (and hence the spatial dimension $n\geq 3$). The main goal of this paper is to prove a Liouville theorem for ancient solutions to \eqref{eqn}, and then use it to study the structure of end time singularities.

The blow-up problem of \eqref{eqn}
has attracted a lot of attention since the work of Fujita \cite{Fujita1966}, see e.g. Giga-Kohn \cite{Giga-Kohn1985, Giga-Kohn2, Giga-Kohn1989} and the monograph of Quittner-Souplet \cite{Quittner-Souplet2019}. By now, the subcritical case ($1<p<\frac{n+2}{n-2}$) has been very well understood.
For example, by Giga-Kohn \cite{Giga-Kohn2} (in the case of positive solutions)  and   Giga-Matsui-Sasayama \cite{Giga04} (in the general case of sign-changing solutions), all interior blow-ups are Type I. That is, if $T$ is the first blow-up time (i.e. $u$ is smooth before $T$), then
\begin{equation}\label{Type I}
   \|u(\cdot, t)\|_{L^\infty(\mathbb{R}^{n})}\leq C(T-t)^{-\frac{1}{p-1}}.
\end{equation}
Notice that this blow-up rate is consistent with the blow-ups of the   ODE solutions to \eqref{eqn},
\begin{equation}\label{ODE solution}
u_T(t)=\pm(p-1)^{-\frac{1}{p-1}}(T-t)^{-\frac{1}{p-1}}.
\end{equation}

In contrast,  if
\[\limsup_{t\to T} (T-t)^{\frac{1}{p-1}}\|u(\cdot, t)\|_{L^\infty(\mathbb{R}^{n})}=+\infty,\]
then it is called Type II blow-ups. In the critical ($p=\frac{n+2}{n-2}$) and supercritical ($p>\frac{n+2}{n-2}$) case,  there do exist Type II finite time blow-up solutions, see \cite{Wei-zhou2020,Del-Musso-Wei2019, Del-Wei-zhou2020,  Velazquez2000, Harada2020(1), Harada2020(2), Schweyer} and \cite{Collot2017-nonradial,Collot-Merle-Raphael2020,Merle-R-S2020-TypeI,Seki2018-JL,Seki2020-Lepin} respectively.

Concerning the supercritical case,
in a series of works \cite{Merle-Matano2004,Merle-Matano2009,Merle-Matano2011}, Matano and Merle systematically study the blow-up problem of \emph{radially symmetric solutions}. By using the intersection number technique,  they are able to give a complete classification of Type I and Type II blow-up behavior (among many other results). See also Mizoguchi  \cite{Mizoguchi1, Mizoguchi2, Mizoguchi3, Mizoguchi4} for more results on the blow-up behavior of radially symmetric solutions. Some of these results will be generalized to the setting \emph{without any symmetry assumption} in this paper.

The main finding of this paper can be summarized as: the ODE  solution $u_T$ plays a special role in  the blow-up problem of \eqref{eqn}. For example, it is stable with respect to base point change, and it contributes to the $(n-1)$ dimensional part in the blow-up set (while the remaining part has a lower Hausdorff dimensional bound). This will be achieved by establishing  a Liouville theorem for ancient solutions of \eqref{eqn}, which gives a characterization of this ODE solution.

\section{Setting and main results}\label{sec main results}
	\setcounter{equation}{0}

\subsection{Setting} 
We  work in the class of \emph{suitable weak solutions} defined in \cite{Wang-Wei2021}. First let us recall the  definition.
\begin{defi}[Suitable weak   solution] \label{definition suitable weak solution} 
Let $B_{1}
$ be the unit ball in $\mathbb{R}^{n}$. A function $u$, defined on the backward parabolic cylinder $Q_1^-:=B_1\times(-1,0)$, is a suitable weak solution of \eqref{eqn}, if $\partial_tu,\nabla u\in L^2(Q_1^-)$, $u\in L^{p+1}(Q_1^-)$, and
		\begin{itemize}
			\item  $u$ satisfies \eqref{eqn} in the weak sense, that is,  for any $\eta\in C_0^\infty(Q_1^{-})$,
			\begin{equation}\label{weak solution I}
				\int_{Q_1^{-}}\left[\partial_t u\eta+\nabla u\cdot\nabla \eta-|u|^{p-1}u\eta\right]dxdt=0;
			\end{equation}
			\item    $u$ satisfies the localized energy inequality: for any $\eta\in C_0^\infty(Q_1^{-})$,
			\begin{equation}\label{energy inequality I}
				\int_{Q_1^{-}}\left[ \left(\frac{|\nabla u|^2}{2}-\frac{|u|^{p+1}}{p+1}\right)\partial_t\eta^2-|\partial_tu|^2\eta^2-2 \eta\partial_tu\nabla u\cdot\nabla\eta\right]dxdt\geq 0;
			\end{equation}
			\item $u$ satisfies the stationary condition: for any $Y\in C_0^\infty(Q_1^{-}, \R^n)$,
			\begin{equation}\label{stationary condition I}
				\int_{Q_1^{-}} \left[\left(\frac{|\nabla u|^2}{2}-\frac{|u|^{p+1}}{p+1}\right)\mbox{div}Y-DY(\nabla u,\nabla u)-\partial_tu \nabla u\cdot Y\right]dxdt=0.
			\end{equation}
		\end{itemize}
 \end{defi}
 A smooth solution satisfies all of these conditions, which can be easily verified by integration by parts. However, a suitable weak solution need not to be smooth everywhere. 
 \begin{defi}[Regular and singular set]
    Given a suitable weak solution $u$, a point $(x_{0}, t_{0})$ is said to be a regular point of $u$ if there exists $0<r<1$ such that $u$ is bounded in $Q_{r}^{-}(x_{0}, t_{0}):=B_{r}(x_{0})\times (t_{0}-r^{2}, t_{0})$. The set containing all the regular points is denoted by $\text{Reg}(u)$. 
   
    The complement of $\text{Reg}(u)$ is the singular set, which is denoted by $\text{Sing}(u)$.
 \end{defi}
 Once   $u\in L^\infty(Q_{r}^{-}(x_0,t_0))$, then by standard parabolic regularity theory, it is a classical solution in $Q_r^{-}(x_0,t_0)$. Therefore for each time slice $t$, the set $\{x:(x, t)\in\text{Reg}(u)\}$ is open and $\{x:(x, t)\in \text{Sing}(u)\}$ is closed. 
Clearly, $(x_{0}, t_{0})\in\text{Sing}(u)$ only if there exists a sequence $(x_{k}, t_{k})\to(x_0,t_0)$ such that $\lim_{k\rightarrow\infty}|u(x_{k}, t_{k})|=+\infty$. So the singular set can also be called \emph{the  blow-up set}.

 Throughout the paper, we denote
 \[  m=2\frac{p+1}{p-1}.\]
And unless otherwise specified, we will always assume that \textbf{\textit{$m$ is not an integer}}.

 \begin{defi}[Ancient solution]
A function $u$ is said to be an ancient solution of the equation \eqref{eqn} if for any $Q_{r}^{-}(x_{0}, t_{0})\subset\mathbb{R}^{n}\times (-\infty, 0)$, $u$ is a suitable weak solution of \eqref{eqn} in $Q_{r}^{-}(x_{0}, t_{0})$.
\end{defi}
\begin{defi}[The class $\mathcal{F}_M$]\label{definition class of F}
Given $M>0$,  $\mathcal{F}_M$ denotes the set of ancient solutions of \eqref{eqn} satisfying the following  Morrey estimate:
 for any  
 $(x_0, t_0)\in\R^{n}\times (-\infty, 0]$  and $r>0$,  
 \begin{equation}\label{Morreyeatimates}
 r^{m-2-n}\int_{Q_r^-(x_0, t_0-r^2)}(|\nabla u|^{2}+|u|^{p+1})dx dt+r^{m-n}\int_{Q_r^-(x_0, t_0-r^2)}
 (\partial_{t}u)^2 dxdt \leq M.
 \end{equation}
\end{defi}
\begin{rmk}
   This Morrey estimate \eqref{Morreyeatimates} is scaling invariant in the following sense: if $u$ satisfies \eqref{Morreyeatimates}, then for any $(x_0,t_0)\in\R^n\times(-\infty,0]$ and $\lambda>0$, 
   \[u_{x_{0}, t_{0},\lambda}(x, t):=\lambda^{\frac{2}{p-1}}u(x_{0}+\lambda x, t_{0}+\lambda^{2}t)\] satisfies the same estimate. (Note that $u_{x_{0}, t_{0},\lambda}$ is also a solution of \eqref{eqn}. This is exactly the scaling invariance of \eqref{eqn}.)

   The condition \eqref{Morreyeatimates} arises from the blow-up analysis for \eqref{eqn}. It is usually deduced from Giga-Kohn's monotonicity formula, cf.  \cite[Proposition 2.2]{Giga-Kohn2} or \cite[Proposition 3.1]{Giga-Kohn1989}. It should be emphasized that there is a time lag in $Q_r^-(x_0, t_0-r^2)$, that is, we must use $t_0-r^2$ instead of $t_0$.  We  notice that a different Morrey space estimate was also proved in Souplet \cite{Souplet2017-Morrey}.
\end{rmk}
\begin{defi}[The class $\mathcal{S}_M$ of backward self-similar solutions]
 A function $w$ defined on $\mathbb{R}^n$, belongs to 
$\mathcal{S}_M$ if $(-t)^{-\frac{1}{p-1}}w(\frac{x}{\sqrt{-t}})\in \mathcal{F}_M$.
\end{defi}
The function $u(x,t):=(-t)^{-\frac{1}{p-1}}w(\frac{x}{\sqrt{-t}})$ satisfies
\[u(\lambda x,\lambda^2t)=\lambda^{\frac{2}{p-1}}u(x,t), \quad \forall \lambda>0,\]
hence the name ``backward self-similar solutions". Because $u$ is a solution to \eqref{eqn}, it is well known that $w$ is a solution to the elliptic equation
\begin{equation}\label{SC1}\tag{SS}
-\Delta w+\frac{y}{2} \cdot\nabla w+\frac{1}{p-1}w=|w|^{p-1}w,\quad\text{in}~\mathbb{R}^{n}.
\end{equation}
Notice that since we require more conditions on suitable weak solutions, for $w\in\mathcal{S}_M$, it also satisfies some integral identities inherited from  \eqref{energy inequality I} and \eqref{stationary condition I}. 

By denoting
\[\kappa=(p-1)^{-\frac{1}{p-1}},\]
it is clear that $0$ and $\pm\kappa$ are the only constant solutions of \eqref{SC1}. The constant solutions $\pm\kappa$ correspond to the ODE solutions $\pm u_0$ defined in \eqref{ODE solution}. As said before, we will show that these constant solutions are very special among solutions of \eqref{SC1}.

\subsection{Liouville theorem}
Now we can state the first main result of this paper, which is a Liouville theorem for ancient solutions.
 \begin{thm}\label{maintheorem1}
Given an ancient solution $u\in \mathcal{F}_{M}$, if there exists $\lambda_{k}\to +\infty$ such that
\begin{equation}\label{main assumption}
\lambda_{k}^{\frac{2}{p-1}}u(\lambda_{k}x,\lambda_{k}^{2} t)\to \kappa(-t)^{-\frac{1}{p-1}} \quad  \text{strongly in} ~~ L_{loc}^{p+1}(\mathbb{R}^{n}\times(-\infty, 0)),
\end{equation}
then there exists a constant $T\geq 0$ such that
 \[u(x, t)=\kappa(T-t)^{-\frac{1}{p-1}},\quad\text{in}~\mathbb{R}^{n}\times(-\infty, 0).\]
 \end{thm}
 \begin{rmk}
\begin{enumerate}
    \item The same result holds for $-\kappa(-t)^{-\frac{1}{p-1}}$.
    \item The function sequence in \eqref{main assumption} is a blow-down sequence. Its strong  convergence in $L_{loc}^{p+1}(\mathbb{R}^{n}\times(-\infty, 0))$ can be established by using the fact that $m$ is not an integer, see Proposition \ref{prop blow-down limit} below.
    \item In \cite{Polacik-Quittner2021-Liouville}, Pol\'a\v cik and Quittner also proved a Liouville theorem for ancient and entire solutions to supercritical Fujita equation, which however is concerned with the characterization of steady states and is restricted to   radially symmetric solutions. 
    \end{enumerate}
 \end{rmk}
 As a special case of Theorem \ref{maintheorem1}, we can prove the following classification result for low entropy ancient solutions.
 \begin{coro}[Low-entropy ancient solutions]\label{Low-entropy ancient solutions}
Assume $\frac{n+2}{n-2}<p<\frac{n+1}{n-3}$ and $u\in\mathcal{F}_{M}$ is a nonnegative ancient solution of \eqref{eqn}. If
\[\sup_{t\in(-\infty, 0)}\lambda(u(\cdot, t))\leq \left(\frac{1}{2}-\frac{1}{p+1}\right)\left(\frac{1}{p-1}\right)^{\frac{p+1}{p-1}},\]
where $\lambda(u(\cdot, t))$ is the entropy defined in \cite{Wang-Wei-Wu}, then either $u=0$ or there exists a constant $T\geq 0$ such that
 \[u(x, t)=\kappa(T-t)^{-\frac{1}{p-1}},\quad\text{in}~\mathbb{R}^{n}\times(-\infty, 0).\]
 \end{coro}

\subsection{Applications: End time singular set}\label{subsection application}
   For finite time blow-up solution of \eqref{eqn},  one question that has received great attention is  the structure and the regularity of the blow-up set $\text{Sing}(u)$, see \cite{Blatt-Struwe, Chou-DU-Zheng2007, Dushizhong2019,Giga-Kohn1989, Merle-Matano2004,  Velazquez1993-2, Zaag-2002, Zaag-2006}. 
  
  \begin{defi}[Blow-up sequence and tangent function]\label{definition of blow-up sequence}
Given a base point $(x_0,t_0)$ and a constant $\lambda>0$, denote
\[u_{x_0,t_0,\lambda}(x,t):=\lambda^{\frac{2}{p-1}}u(x_0+\lambda x, t_0+\lambda^2 t).\]
If there exists a sequence  $\lambda_k\to0^+$ and a function $w_{0}\in\mathcal{S}_{M}$ such that
  \[u_{x_0,t_0,\lambda_k}\to (-t)^{-\frac{1}{p-1}}w_{0}\left(\frac{x}{\sqrt{-t}}\right) \quad\text{in}~ L^{p+1}_{loc}(\R^n\times(-\infty,0)),\]
  then $w_0$ is called a tangent function of $u$ at $(x_0,t_0)$.
  
The set of tangent functions at $(x_0,t_0)$ is denoted by  $\mathcal{T}(x_0,t_0;u)$. 
   \end{defi}
Blow-down sequence can be defined in the same way by using sequences  $\lambda_k\to+\infty$. The convergence of blow-up sequences and the self-similarity of their limits can be established by compactness arguments and monotonocity formula, see Section \ref{sec preliminary}.

A more familiar form used since the work of Giga-Kohn \cite{Giga-Kohn1985} is to reformulate the equation \eqref{eqn} in self-similar coordinates.
\begin{defi}[Self-similar transform]
 Given a suitable weak ancient solution $u$ of \eqref{eqn}, its self-similar transform with respect to  a base point $(x_0,t_0)$ is   
\begin{equation}\label{definition of self-similar transform}
w(y, s)=(t_0-t)^{\frac{1}{p-1}}u(x, t),\quad 
y=(t_0-t)^{-\frac{1}{2}}(x-x_0),\quad s=-\log(t_0-t).
\end{equation}
\end{defi}
The function $w$ obtained by taking self-similar transform is a solution of
\begin{equation}\label{self-similar eqn}\tag{RF}
\partial_{s}w=\Delta w-\frac{y}{2}\cdot\nabla w-\frac{1}{p-1}w+|w|^{p-1}w,\quad\text{in}~\mathbb{R}^{n}\times (-\infty, +\infty).
\end{equation}
Notice that \eqref{SC1} is just the steady version of \eqref{self-similar eqn}. The convergence of blow-up sequence to tangent functions in Definition \ref{definition of blow-up sequence} is equivalent to the convergence of $w(\cdot, s)$ as $s\to+\infty$, and the convergence of blow-down sequence corresponds to the convergence as $s\to-\infty$.  Just as in the subcritical case considered in Giga-Kohn \cite{Giga-Kohn1985}, this convergence is
usually deduced from  the gradient flow structure of \eqref{self-similar eqn}:
For any $w\in H^1_w(\R^n)\cap L^{p+1}_w(\R^n)$, let
\begin{equation}\label{energy}\tag{E}
E(w)= \int_{\mathbb{R}^{n}} \left[\frac{1}{2} |\nabla w|^2  +\frac{1}{2(p-1)} w
^2  - \frac{1}{p+1} |w|^{p+1} \right]\rho dy
\end{equation}
be the energy of $w$. Then \eqref{self-similar eqn} is the gradient flow of this energy functional, while \eqref{SC1} is its Euler-Lagrange equation.

   Theorem \ref{maintheorem1} allows us to prove the following result.
\begin{thm}[Uniqueness and stability of ODE blow-ups]\label{maintheorem2}
Let $u$ be a suitable weak solution of \eqref{eqn} in $Q_{1}^{-}$, satisfying the Morrey estimate \eqref{Morreyeatimates} for all backward parabolic cylinders of the form $Q_r^-(x,t-r^2)\subset Q_1^-$, where $(x,t)\in Q_1^-$.
 
 If   $\kappa\in \mathcal{T}(0,0; u)$, then the followings hold.
\begin{enumerate}[label=(\Alph*)]
   \item\label{item:uniqueness} Uniqueness of blow-up limit: $\mathcal{T}(0,0; u)=\{\kappa\}$;
    \item\label{item:stability} Stability of ODE blow-up: there exists $\delta>0$ such that for any $x\in B_{\delta}(0)\cap \text{Sing}(u)$,  $\mathcal{T}(x,0;u)=\{\kappa\}$;
     \item\label{item:positivity} Positivity preserving: $u>0$ in $Q_\delta^-(0,0)$;
    \item\label{item:Type I} Type I blow-up rate:  $u$ is smooth in $Q_\delta^-(0, 0)$ and there exists a constant $C>0$ such that
    \begin{equation}\label{Type I blow-up rate}
    u(x,t)\leq C(-t)^{-\frac{1}{p-1}}, \quad \forall (x,t)\in Q_\delta^-(0, 0).
    \end{equation}
\end{enumerate}
\end{thm}
\begin{rmk}
\begin{enumerate}
    \item 

Item \ref{item:uniqueness} in Theorem \ref{maintheorem2} justifies rigorously the discussions in \cite{Merle-Matano2004}.  (See the 
discussions below Remark 2.4 therein.) Different from \cite{Merle-Matano2004}, we do not need  any assumption on the symmetry of solutions.

\item The  uniqueness and stability property of ODE blow-ups can be understood in the following way: if a solution is close to the ODE solution at one scale, then it is even closer to the ODE solution on smaller scales, and this not only holds with respect to the original base point, but also with respect to all nearby singular points. 

This is very different from existing reulsts on the stability of Type I blow-ups in the literature, e.g. the one in Collot-Raphael-Szeftel \cite{Collot-Raphael-Szeftel2019}. Those are mostly on the functional analysis side, where the equation is viewed as a dynamical system in some function spaces, and  then a kind of dynamical stability of steady states (i.e. solutions of \eqref{SC1}) is established. In contrast, in the above stability result, we do not impose any perturbation on initial values. (In fact, there is no initial value because we are only concerned with solutions in backward parabolic cylinders.)

This stability property is closer to the result for  generic singularities in mean curvature flows, see Colding-Minicozzi \cite{Colding-Minicozzi2016}.  Just as there, the above theorem says that ODE solutions are generic among all blow-ups in \eqref{eqn}. Furthermore, the proof of our results also rely on some ideas inspired by Colding-Minicoozi \cite{Colding-Minicozzi2012} and Colding-Ilmanen-Minicozzi \cite{Colding-I-M2015}.

\item
The stability property also occurs in  mean curvature flows (see \cite{Felix-Natasa, Sun-Xue}). In \cite[Theorem 1.7]{Felix-Natasa},  Schulze and Sesum proved the ‌stability of neckpinch singularities. But it seems that the mechanism behind them are very different, see Subsection 2.4 for more discussions.
\end{enumerate}
\end{rmk}

By the above theorem, the following subsets of $\text{Sing}(u)$ are well defined,
\[ \mathcal{R}^\pm:=\{x: \mathcal{T}(x,0; u)=\{\pm\kappa\}\}.\]
Our last main result is about the structure of $\mathcal{R}^\pm$ and its complement $\text{Sing}(u)\setminus (\mathcal{R}^+\cup\mathcal{R}^-)$. 
\begin{thm}\label{maintheorem3}
Let $u$ be a solution of the equation \eqref{eqn} satisfying the assumptions in Theorem \ref{maintheorem2}. If $\text{Sing}(u)\cap \{t=0\}\neq\emptyset$ and 
$u$ is not an ODE solution, 
  then 
\begin{enumerate}[label=(\Alph*)]
  \item\label{item:openness} both $\mathcal{R}^+$ and $\mathcal{R}^-$ are relatively open in $\{x:(x,0)\in
  \text{Sing}(u)\}$;
  \item\label{item:rectifibility} $\mathcal{R}^\pm$ are  $(n-1)$-rectifiable sets; 
  \item\label{item:Hausdorff} the Hausdorff dimension of $\{x:(x,0)\in
  \text{Sing}(u)\}\setminus (\mathcal{R}^+\cup\mathcal{R}^-)$ is at most $n-[m]-1$, where $[m]$ is the integer part of $m$.
   \end{enumerate}
\end{thm}
\begin{rmk}[End point singularity]
In the above theorem, the Hausdorff dimension estimate in \ref{item:Hausdorff} is the same with the one in \cite[Theorem 7.1]{Wang-Wei2021}. However, that theorem is 
concerned with singular sets in the full parabolic cylinder $Q_1:=B_1\times(-1,1)$, while  the above theorem is about the singular set at the end time.  This leads to the difference related to the two additional parts $\mathcal{R}^\pm$ in $\{x:(x,0)\in
  \text{Sing}(u)\}$. The above theorem says that, these $(n-1)$ dimensional parts arise just because the tangent function is an ODE solution. In other words, if a solution $u$ can be continued further in the time, then the tangent function at a singular point cannot be ODE solutions.\footnote{This is because ODE solutions do not satisfy the Morrey space estimate \eqref{Morreyeatimates} for any base point $(x_0,t_0)$ with $t_0>0$. } Once this obstacle does not exist, then Federer dimension reduction principle can be applied to lower the Hausdorff dimension estimate down to $n-[m]-1$.

In view of this fact, it  seems natural to conjecture that a solution to \eqref{eqn} in $Q_1^-$ can be continued beyond time $0$ if and only if these ODE blow-ups does not appear at any singular point on $B_1\times\{0\}$.
\end{rmk}

Finally, we present three more applications of Theorem \ref{maintheorem2} to ancient solutions or backward self-similar solutions.
\begin{coro}\label{coro second Liouville}
For any  $u\in\mathcal{F}_{M}$, if
$u$ blows up on the entire $\R^{n}\times \{0\}$, then $u(x,t)\equiv \kappa(-t)^{-\frac{1}{p-1}}$.
\end{coro}

\begin{coro}\label{coro nearly subcritical case}
If $\frac{n+2}{n-2}<p<\frac{n+1}{n-3}$ and $w\in\mathcal{S}_{M}$ is a  smooth solution of \eqref{SC1}, then $w$ is bounded on $\R^n$.
\end{coro}
\begin{coro}\label{coro radially symmetric cae}
If $w$ is a smooth radially symmetric solution of \eqref{SC1}, then either $w$ is a constant solution or there is a positive constant $C$ such that
\begin{equation}\label{decay estimate for radially symmetric soluitons}
    |w(y)|\leq C(1+|y|)^{-\frac{2}{p-1}},\quad\text{in}~\mathbb{R}^{n}.
\end{equation}
\end{coro}
Corollary \ref{coro second Liouville} can be viewed as a local version of the assertion that complete blow-up (that is, solution which does not have continuation beyond the blow-up time) is totally caused by ODE blow-ups.

\subsection{Difficulties and approaches} 

 Theorem \ref{maintheorem1} is a generalization of the following Liouville theorem of Merle-Zaag \cite[Corollary 1.6]{Merle-Zaag} to the supercritical setting. 
 \begin{thm}[Merle-Zaag \cite{Merle-Zaag}]
If $1<p<\frac{n+2}{n-2}$ and $u$ satisfies
 \[0\leq u(x, t)\leq C(-t)^{-\frac{1}{p-1}},\quad\text{in}~\mathbb{R}^{n}\times(-\infty, 0),\]
 then either $u\equiv 0$ or there exists a constant $T\geq 0$ such that
 \[u(x, t)=\kappa(T-t)^{-\frac{1}{p-1}},\quad\text{in}~\mathbb{R}^{n}\times(-\infty, 0).\] 
 \end{thm}
The proof in \cite{Merle-Zaag} can be summarized as follows:
\begin{description}
    \item[Step 1] Define the self-similar transform $w$ as in \eqref{definition of self-similar transform}.
Then $w$ is a \emph{bounded solution} of  \eqref{self-similar eqn}.

\item[Step 2] Because  $p$ is subcritical, by Giga-Kohn  \cite{Giga-Kohn1985}, $w_{\pm\infty}(y)=\lim_{s\to\pm\infty} w(y, s)$ exists and they are bounded solutions of \eqref{SC1}. By Giga-Kohn's Liouville theorem, they could only be $0$ or $\kappa$. By Giga-Kohn's monotonicity formula,  they can exclude the possibility that 
$(w_{+\infty}, w_{-\infty})=(0, \kappa)$. Moreover, if
$(w_{+\infty}, w_{-\infty})=(0, 0)$ or $(w_{+\infty}, w_{-\infty})=(\kappa, \kappa)$, then $w$ is a constant solution.

\item[Step 3] For the remaining possibility $(w_{+\infty}, w_{-\infty})=(\kappa, 0)$,  they  linearize $w$ around the constant solution by defining $v=w-\kappa$. Then they proved the famous Merle-Zaag lemma  (see \cite[Lemma A.1]{Merle-Zaag}), and used this to classify the behavior of $v$ as $s\to-\infty$.

\item[Step 4] By combining the blow-up criteria and some geometrical transformation, they can finish the proof.
\end{description}

In our setting, the exponent $p$ is supercritical and the  solution is only a suitable weak ancient solution, so the following difficulties will be encountered.

First, we can still take the same self-similar transform as in Step 1, but now $w$ is not a \emph{bounded solution} of \eqref{self-similar eqn}. So  even if $w_{\pm\infty}(y)=\lim_{s\to\pm\infty} w(y, s)$  exists, the convergence holds only in some weak sense, and a priori $w_{\pm\infty}$ may not be smooth. 

Second,  it seems impossible to get any classification for solutions to \eqref{SC1}, so  an assumption on the blow-down limit, that is, \eqref{main assumption}, is necessary.

It turns out that under the assumption \eqref{main assumption}, the above difficulties can all be overcome. The main reason is: \eqref{main assumption} says that the ancient solution $u$ looks like the ODE solution $u_0$ at large scales, so we can perform \emph{a linearization analysis} around $u_0$. However, to implement this idea, in particular, to circumvent the difficulty by the a priori irregularity of $u$, we need to introduce \emph{a key auxiliary function} $RT(x)$. This function can be roughly understood as  the first time when $u$ is not close to  $u_0$. By using this function together with standard blow-up arguments, we can get  the same setting (e.g. a bounded solution $w$ of \eqref{self-similar eqn}) as in the above subcritical case.

To define $RT$, first, we establish a rigidity theorem, which  says
$\kappa$ is isolated among solutions to \eqref{SC1}. This then implies that the ODE solution $u_0$ is the \emph{unique} blow-down limit.  This rigidity result is inspired by the rigidity theorem for self-shrinkers to mean curvature flow in  Colding-Ilmanen-Minicozzi \cite{Colding-I-M2015}. We point out that in mean curvature flows, the entropy and the Gaussian density is nonnegative and the second fundamental form of self-shrinkers satisfies the same equation as its mean curvature  (see \cite[Proposition 4.5]{Colding-I-M2015}), but these  do not have natural counterparts for \eqref{eqn}. Therefore, in each step of our proof, we need to  find suitable substitutes. 

Because $u_0$ is a smooth ancient solution of \eqref{eqn}, uniqueness of the blow-down limit
implies that its self-similar transform, $w$ is locally bounded for all $-s\gg 1$, and it converges smoothly to $\kappa$ as $s\to-\infty$. This then allows us to introduce \emph{the key auxiliary function} $RT(x)$, which can be roughly defined as the first time the condition
\[|\Delta u(x,t)|<\frac{1}{2}u(x,t)^p\]
is violated, that is, the first time when $u$ is not close to the ODE solution $u_0$. 

If $RT$ is  bounded on $\R^n$, then we can show that the self-similar transform $w$ is bounded on $\R^n\times (-\infty,-T_0)$ for some $T_0>0$. This allows us to apply Merle-Zaag lemma as in Step 3 above to deduce that $w$ (hence $u$) is an ODE solution.

Of course, $RT$ could also be unbounded. This needs us to perform   rescaling one more time, which then gives a contradiction with the above result in the bounded $RT$ case.

For mean curvature flows, Choi-Haslhofer-Hershkovits \cite{Choi-H-H2022} and  Choi-Haslhofer-Hershkovits-White \cite{Choi-H-H-W2022} have established  some Liouville properties  for ancient solutions. These results give a characterization of   cylinders (as self-shrinkers to mean curvature flows), and our  Liouville theorem has a formal correspondence with these results. But it seems that the  nature of these two Liouville properties are very different. 
In fact, the ODE solution $u_0$ in our problem should correspond to spheres (as self-shrinkers to mean curvature flows), but it  also has  many similarities  with cylinders. For example,  some form of Merle-Zaag lemma are all needed in the proof of these Liouville results.

Finally,  the assumption that $m$ is not an integer is very important  for our proof. This is because we need a crucial compactness properties for suitable weak solutions of \eqref{eqn}, see Proposition \ref{prop compactness} below. In fact, if $m$ is an integer, the class $\mathcal{F}_{M}$  is not compact in $L^{p+1}_{loc}$ or $H^1_{loc}$. This is because   we can not apply Marstrand theorem (see \cite[Theorem 1.3.12]{Lin-Yang-book}) to prove that some defect measures \footnote{The precise definition can be found in \cite[Section 4]{Wang-Wei2021}} are trivial. We believe the analysis of this case needs more work.

This paper is organized as follows. In Section 3 and Section 4, we collect some preliminary results. In Section 5, we  prove a rigidity result about the constant solution $\kappa$. In Section 6,  this rigidity result is used to show the unique asymptotic behavior of $u$ under the assumptions of Theorem \ref{maintheorem1}. Based on this asymptotic behavior, we introduce the auxiliary function $RT$ in Section 7. In Section 8, we  prove a conditional Liouville theorem for solutions with bounded $RT$. In Section 9 and Section 10, we will combine the previous results to give the proof of Theorem \ref{maintheorem1} and Theorem \ref{maintheorem2} respectively. In Section 11, we discuss several applications of Theorem \ref{maintheorem2}. Finally, the proof of several technical results in the above process are given in appendices.

\vspace{2mm}

\textbf{Notation.}
\begin{itemize}
\item Throughout this paper, we  will use $C$ (large) and $c$ (small) to denote universal constants which depend only on $n$, $p$ and $M$. They could be different from line to line.

\item  For $(y, t)\in\mathbb{R}^{n}\times (0, \infty)$, we denote the standard heat kernel by
\[G(y, t)=(4\pi t)^{-\frac{n}{2}}\exp\left(-\frac{|y|^{2}}{{4t}}\right).\]
We also set
\[\rho(y):=G(y,1)=(4\pi)^{-\frac{n}{2}}\exp\left(-\frac{|y|^{2}}{{4}}\right)\]
to be the standard Gaussian.

\item The open ball in $\mathbb{R}^{n}$ is denoted by $B_{r}(x)$.
 The backward parabolic cylinder is $Q_{r}^{-}(x, t)=B_{r}(x)\times (t-r^{2}, t)$.
    If the center is the origin, it will not be written down explicitly.

\item For $1<q<+\infty$, we define the weighted spaces
\[L_{w}^{q}(\R^{n})=\left\{f\in L_{loc}^{q}(\R^{n}):\int_{\R^{n}}|f|^{q}e^{-\frac{|y|^{2}}{4}} dy<+\infty\right\}\]
and
\[H_{w}^{1}(\R^{n})=\left\{f\in H_{loc}^{1}(\R^{n}):\int_{\R^{n}}[|f|^{2}+|\nabla f|^{2}]e^{-\frac{|y|^{2}}{4}} dy<+\infty\right\}.\]

\item We use standard function space notations in the parabolic setting, e.g. $C^{2,1}$ denotes the space of functions which are $C^2$ in $x$ and $C^1$ in $t$.
\end{itemize}

\section{Preliminaries}\label{sec preliminary}
	\setcounter{equation}{0}
 
In this section, we  collect several preliminary results, which will be used frequently later.

The first one is the monotonicity formula for \eqref{eqn}. 
\begin{lem}[Monotonicity formula]\label{Monotonicity formula}
 For any $u\in\mathcal{F}_M$ and $(x,t)\in\R^n\times(-\infty,0]$, the quantity
\begin{align}
E(s; x, t, u):=&\frac{1}{2}s^{\frac{p+1}{p-1}}\int_{\mathbb{R}^{n}}|\nabla u(y, t-s)|^{2}G(y-x, s)dy\notag\\
&-\frac{1}{p+1}s^{\frac{p+1}{p-1}}\int_{\mathbb{R}^{n}}|u(y, t-s)|^{p+1}G(y-x, s)dy\\
&+\frac{1}{2(p-1)}s^{\frac{2}{p-1}}\int_{\mathbb{R}^{n}}u(y, t-s)^2G(y-x, s)dy\notag
\end{align}
is a non-decreasing function of $s$ in $(0,+\infty)$.

Moreover, if $E(s;x,t,u)\equiv \text{const.}$, then $u$ is backward self-similar with respect to $(x,t)$.
\end{lem}
This is a reformulation of Giga-Kohn's monotonicity formula in \cite[Proposition 3]{Giga-Kohn1985}. For classical solutions,  this monotonicity formula is usually proved by integration by parts, which unfortunately does not work for weak solutions. The notion of suitable weak solutions in Definition \ref{definition suitable weak solution} is designed exactly for this purpose. By substituting suitable test functions into  the three conditions in that definition\footnote{For example, we can plug the vector field $Y(x)=\eta(x)x$ into the stationary condition \eqref{stationary condition I}, where $\eta$ is a standard cut-off function.}, and then passing to the entire space $\R^n$ \footnote{Here we also need the Morrey space estimate from the definition of $\mathcal{F}_M$, Definition \ref{definition class of F}. This condition implies that $u$ is bounded in averaged sense, and it guarantees various error terms from cut-off functions in this calculation are negligible so that we can pass to $\R^n$.}, we get the same result as in the classical case (but understood in the distributional sense).
  
Sometimes it is convenient to consider an averaged version of $E$, that is,
\[\bar{E}(s; x, t, u)=s^{-1}\int_{s}^{2s}E(\tau; x, t, u)d\tau.\]
This is because $\bar{E}$ involves space-time integrals, and hence it is continuous in $(x,t)$.

We also need an  $\varepsilon$-regularity theorem on \eqref{eqn}, see Du  \cite[Theorem 3.1]{Dushizhong2019} and  also Miura-Takahashi \cite[Section 4]{Takahashi}.
\begin{prop}[$\varepsilon$-regularity]\label{epsilonregularity}
There exist universal constants $\varepsilon_{0}$ and $C_{0}$ so that the following holds. If there exist $s, \delta$  with $s\geq 
\delta$ such that for any $(x, t)\in Q_{2\delta}^-(x_{0}, t_{0})$,
\[\bar{E}(s; x, t, u)\leq\varepsilon_{0},\]
then
\[\sup_{Q^{-}_{\delta}(x_{0}, t_{0})}|u|\leq C_{0}\delta^{-\frac{2}{p-1}}.\]
\end{prop}

The next result 
is a compactness property of $\mathcal{F}_{M}$, whose proof can be found in \cite[Theorem 6.1]{Wang-Wei2021}. This is the main reason that we only consider the case when $m$ is not an integer.
\begin{prop}[Compactness of $\mathcal{F}_M$]\label{prop compactness}
For any sequence $u_k\in\mathcal{F}_M$,  after passing to a subsequence, there exists a function $u_\infty\in\mathcal{F}_M$ such that
\begin{itemize}
    \item  $u_k$ converges strongly to  $u_\infty$ in $L^{p+1}_{loc}(\R^n\times(-\infty,0))$;
    \item $\nabla u_k$ converges strongly to   $\nabla u_\infty$ in $L^2_{loc}(\R^n\times(-\infty,0))$;
    \item $\partial_tu_k$ converges weakly to   $\partial_tu_\infty$ in $L^2_{loc}(\R^n\times(-\infty,0))$.
  \end{itemize}
\end{prop}
\begin{rmk}
    Throughout the paper,  if a sequence $u_k\in\mathcal{F}_M$ converges to a limit $u_\infty$ in the above sense, then we say $u_k\to u_\infty$. Notice that in Proposition \ref{prop compactness}, the strong convergence of $\partial_{t}u_{k}$ is not claimed.
\end{rmk}
\begin{coro}\label{coro convergence of singular points}
Given a sequence of  $u_k\in \mathcal{F}_M$ and $(x_k,t_k)\in \R^n\times(-\infty,0]$, if $u_k\to u_\infty$, $(x_k,t_k)\to (x_\infty, t_\infty)$ and $(x_k,t_k)\in\text{Sing}(u_k)$, then  $(x_\infty,t_\infty)\in\text{Sing}(u_\infty)$.

Equivalently, if $(x_\infty, t_\infty)\in \text{Reg}(u_\infty)$, then there exists a $\delta>0$ such that 
for all $k$ large, $u_k\in C^{2,1}(Q_\delta^-(x_\infty,t_\infty))$ and it converges to $u_\infty$ in this space.
\end{coro}
\begin{proof}
    This follows by combining Proposition \ref{epsilonregularity} with the continuity of $\bar{E}$ with respect to  both $(x,t)$ and $u$ (using the topology specified in Proposition \ref{prop compactness}).
\end{proof}

\begin{prop}[blow-down/up limits]\label{prop blow-down limit}
For any $u\in\mathcal{F}_M$, $(x_0,t_0)\in\mathbb{R}^{n}\times(-\infty, 0]$ and sequence $\{\lambda_{k}\}$ tending to $+\infty$ as $k\to\infty$, denote the blow-down  sequence by
\begin{equation} \label{revise 1}
u_{k}(x,t):= \lambda_{k}^{\frac{2}{p-1}} u(x_0+\lambda_{k} x, t_0+\lambda_{k}^2 t).
\end{equation}
Then up to subsequences, $u_k\to u_\infty$ as $k\to+\infty$, where $u_\infty\in\mathcal{F}_M$.  Moreover, there exists a function $w_{\infty}\in\mathcal{S}_{M}$ such that $u_{\infty}(x, t)=(-t)^{-\frac{1}{p-1}}w_{\infty}(\frac{x}{\sqrt{-t}})$.

If $\{\lambda_{k}\}$ tends to $0$ as $k\to\infty$ and we define blow-up sequence as in \eqref{revise 1}, then the same results hold.
\end{prop}
\begin{proof}
This is essentially \cite[Lemma 5.2]{Wang-Wei2021}.
\end{proof}

At this stage, it is worth to point out that in the above  result, $u_{\infty}$ is obtained by a compactness argument, so it is not clear if they are independent of the choice of subsequences.

Finally, we prove some uniform integral estimates for functions in $\mathcal{F}_{M}$ or $\mathcal{S}_M$.
\begin{lem}\label{pre-estimate}
There exists a universal, positive constant $C$    such that for any $w\in\mathcal{S}_M$, 
\[\int_{\mathbb{R}^{n}}(|\nabla w|^{2}+|w|^{p+1})\rho dy\leq C.\]
In other words,  if   $w\in\mathcal{S}_{M}$, then $w\in H_{w}^{1}(\mathbb{R}^{n})\cap L_{w}^{p+1}(\mathbb{R}^{n})$.
\end{lem}
\begin{proof}
Denote $u(x, t):=(-t)^{-\frac{1}{p-1}}w(\frac{x}{\sqrt{-t}})$, then $u\in\mathcal{F}_{M}$. For each $\lambda\geq 1$ and $x_{0}\in B_{2\lambda}$, it follows from \eqref{Morreyeatimates} that
\[\int_{Q_{1}^{-}(x_{0}, -1)}|u|^{p+1}dxdt\leq M.\]
Notice that there exists a positive constant $C(n)$ such that we can cover $B_{2\lambda}\times (-2, -1)$
 by at at most $C(n)\lambda^{n}$ backward parabolic cylinders of size $1$, then
\begin{equation*}\label{2.1}
\int_{-2}^{-1}\int_{|x|\leq 2\lambda}|u|^{p+1}dxdt\leq C\lambda^{n},
\end{equation*}
where $C$ is a universal positive constant. This implies that
\begin{equation*}
\int_{-2}^{-1}\int_{|x|\leq \sqrt{-t}\lambda}|u|^{p+1}dxdt\leq C\lambda^{n}.
\end{equation*}
By the definition of $u$, this integral can be written as
\[\int_{-2}^{-1}\int_{|x|\leq\sqrt{-t}\lambda}|u|^{p+1}dxdt=\int_{-2}^{-1}(-t)^{-\frac{p+1}{p-1}+\frac{n}{2}}dt\int_{|y|\leq\lambda}|w|^{p+1}dy.\]
Since $p$ is supercritical, we have $-\frac{p+1}{p-1}+\frac{n}{2}>0$. Hence
\[(-t)^{-\frac{p+1}{p-1}+\frac{n}{2}}\geq 1, \quad\text{in}~(-2, -1).\]
Therefore, there exists a universal constant $C$ such that
\[\int_{|y|\leq\lambda}|w|^{p+1}dy\leq C\lambda^{n}.\]
By this estimate, we calculate
\[\begin{aligned}
\int_{\mathbb{R}^{n}}|w(y)|^{p+1}(4\pi )^{-\frac{n}{2}}e^{-\frac{|y|^{2}}{4}}dy=&\sum_{j=0}^{\infty}\int_{j\leq |y|\leq j+1}|w(y)|^{p+1}(4\pi )^{-\frac{n}{2}}e^{-\frac{|y|^{2}}{4}}dy\\
\leq&\sum_{j=0}^{\infty}C(j+1)^{n}e^{-\frac{j^{2}}{4}} \leq C.
\end{aligned}\]
The integral estimate on $\nabla w$ can be got in the same way.
\end{proof}
As a consequence of Lemma \ref{Monotonicity formula}, Proposition \ref{prop blow-down limit} and Lemma \ref{pre-estimate}, we can show
\begin{coro}\label{coro 3.7}
    There exists a universal constant $C$ such that for any $u\in\mathcal{F}_M$, $(x,t)\in\R^n\times(-\infty,0]$ and $s>0$, 
    \begin{equation}\label{bound on monotonicity formula}
        0\leq E(s;x,t,u)\leq\bar{E}(s;x,t,u)\leq C.
    \end{equation}
\end{coro}
In \eqref{bound on monotonicity formula}, the first inequality is    a   consequence of \cite[Proposition 2.1]{Merle-Zaag2000}. In fact, by the $\varepsilon$-regularity (Proposition \ref{epsilonregularity}), if $(x,t)$ is a regular point, then 
\[\lim_{s\to0^+}E(s;x,t,u)=0,\]
while if $(x,t)$ is a singular point, then 
\[E(s;x,t,u)\geq \varepsilon_0, \quad \forall s>0.\]
\begin{lem}\label{control outside region}
Given $0<\delta<1/16$ and $\epsilon>0$, there exists a constant $R$ depending only on $n,M, p,\delta, \epsilon$  such that if $u\in\mathcal{F}_{M}$, then
\begin{equation}\label{control outside region1}
\int_{-1-\frac{\delta}{2}}^{-1+\frac{\delta}{2}}\int_{\mathbb{R}^{n}\backslash B_{R}}(-t)^{\frac{p+1}{p-1}}|u|^{p+1}G(x, -t)dxdt<\epsilon.
\end{equation}
\end{lem}
This lemma is also a consequence of the Morrey estimate in \eqref{Morreyeatimates}. Since the detail is similar to the proof of Lemma \ref{pre-estimate}, we will omit it.

 \section{A characterizations of the constant solutions}
	\setcounter{equation}{0}
 
Recall that $\pm\kappa$ and $0$ are the only   constant solutions of \eqref{SC1}. For any  solution of \eqref{SC1}, define
\begin{equation}\label{definefunctioneta}
\Lambda(w)(y)=\frac{2}{p-1}w(y)+ y\cdot\nabla w(y).
\end{equation}
The constant solutions have the following characterization (see \cite[Proposition 5.1]{Wang-Wei-Wu}).
\begin{prop}\label{ChCS}
For any bounded solution  $w$ of \eqref{SC1},  it is a constant  if and only if  $\Lambda(w)$ 
does not change sign in $\mathbb{R}^{n}$.
\end{prop}
The constant solutions can also be characterized by their energy. Denote
\[\mathcal{B}_{n}=\{w: w~\text{is a nonzero bounded solution of}~\eqref{SC1}\}.\]
Since $\kappa\in\mathcal{B}_n$, $\mathcal{B}_n\neq\emptyset$. The main result of \cite{Wang-Wei-Wu} gives the following characterization of the constant solutions $\pm\kappa$ in $\mathcal{B}_n$.
\begin{thm}\label{ME}
If $w\in\mathcal{B}_{n}$ and $w\neq\pm\kappa$, then
\[E(w)>E(\kappa).\]
\end{thm}
\begin{rmk}
    We believe this characterization for constant solutions should hold among the larger class $\mathcal{S}_M$, not only bounded solutions. This is Conjecture 1 in \cite[Section 12]{Wang-Wei-Wu}, see also \cite[Theorem 1.7]{Wang-Wei-Wu} for partial results. If this characterization is true, then Theorem \ref{maintheorem1} will follow rather directly. On the other hand, notice that Theorem \ref{maintheorem1} implies that there is no non-constant connecting orbit of \eqref{self-similar eqn} starting from $\kappa$ (that is, a solution $w(y, s)$ satisfying $\lim_{s\to-\infty}w(y,s)=\kappa$). 
\end{rmk}

\section{Rigidity of the constant solutions}
	\setcounter{equation}{0}

In this section, we prove two  rigidity results for the constant solutions $\pm\kappa$.  The first rigidity theorem says that $\kappa$ is isolated among all solutions to \eqref{SC1} with respect to the $L_{w}^{2}(\mathbb{R}^{n})$ distance.
\begin{thm}\label{thm rigidity 1}
There exists a positive constant $\varepsilon_1$ such that if $w\in\mathcal{S}_{M}$ and
\begin{equation}\label{L2 close}
\|w-\kappa\|_{L^2_w(\R^n)}\leq \varepsilon_1,
\end{equation}
then $w\equiv \kappa$.
\end{thm}

To state the second rigidity theorem, we need a notation.
For any $R>1$, if  $w_{1}, w_{2}$ are two  continuous functions  on $\overline{B_{R}(0)}$,   define
\[\text{dist}_{R}(w_{1}, w_{2})=\sup_{y\in B_{R}(0)}|w_{1}(y)-w_{2}(y)|.\]
The second rigidity result says $\kappa$ is isolated with respect to this distance. We emphasize  that this  rigidity result holds for any $p>1$, that is, it does not need the restriction that $m=2\frac{p+1}{p-1}$ is not an integer.
\begin{thm}\label{thm rigidity 2}
There exist  $\varepsilon_2$ and $R_\ast$ depending only on $n, p, M$ such that if $w\in\mathcal{S}_M$, it is smooth in $\overline{B_{R_\ast}(0)}$ and
\begin{equation}\label{rigidity condition1}
\text{dist}_{R_\ast}(w,\kappa)\leq \varepsilon_2,
\end{equation}
then $w\equiv \kappa$.
\end{thm}

\begin{rmk}
\begin{enumerate}
    \item A consequence of the $\varepsilon$-regularity result (Proposition \ref{epsilonregularity}) is that  among solutions to \eqref{SC1},   $0$  is isolated. The above theorem says that the other two constant solutions, $\pm\kappa$, are also isolated.
    \item Theorem \ref{thm rigidity 2} says that if a self-similar solution  is close to $\kappa$ in a sufficiently large ball, then it must be $\kappa$ itself. This holds even when the self-similar solution is not presumed to be smooth outside of the ball. This fact is important in applications.
\end{enumerate}
\end{rmk}

The first rigidity result  will follow from the second one together with some compactness properties of $\mathcal{S}_M$. Theorem \ref{thm rigidity 2} is inspired by the corresponding result for self-shrinkers to mean curvature flows in Colding-Ilmanen-Minicozzi  \cite{Colding-I-M2015}.
The proof also follows the iteration arguments developed there. This iteration relies on the following two propositions.
\begin{prop}\label{pro2}
There exists a universal constant $R_{1}$ such  that  for any $w\in\mathcal{S}_M$, if $R>R_{1}$, $w$ is smooth in $\overline{B_R(0)}$ and 
\[\text{dist}_{R}(w,\kappa)+\text{dist}_{R}\left(\Lambda(w), \frac{2}{p-1}\kappa\right)\leq \frac{\kappa}{p+1},\]
 then
\[\text{dist}_{R-3}(w,\kappa)+\text{dist}_{R-3}\left(\Lambda(w), \frac{2}{p-1}\kappa\right)\leq\frac{\kappa}{10(p+1)}.\]
\end{prop}
\begin{prop}\label{pro1}
There exist two universal constants $R_2$ and $\theta$ such  that  for any $w\in\mathcal{S}_M$, if $R>R_2$, $w$ is smooth in $\overline{B_R(0)}$ and 
\[\text{dist}_{R}(w,\kappa)+\text{dist}_{R}\left(\Lambda(w), \frac{2}{p-1}\kappa\right)\leq \frac{\kappa}{10(p+1)},\]
 then $w$ is smooth in $\overline{B_{(1+\theta)R}(0)}$ and
\[\text{dist}_{(1+\theta)R}(w,\kappa)+\text{dist}_{(1+\theta)R}\left(\Lambda(w), \frac{2}{p-1}\kappa\right)\leq\frac{\kappa}{p+1}.\]
\end{prop}

Proposition \ref{pro2} shows that a worse bound on a sufficiently large ball $B_R(0)$ implies a better bound on a slightly smaller ball $B_{R-3}(0)$.  Proposition \ref{pro1} shows  that a better bound on a large ball $B_R(0)$ implies a worse bound on $B_{(1+\theta)R}(0)$ for some fixed $\theta>0$. Here the key point is that the constant $\theta$ is independent of $R$. 

The proof of Proposition \ref{pro2} and Proposition \ref{pro1} will be given in Appendix \ref{sec proof of Proposition 4.4} and Appendix \ref{sec proof of proposition 4.5} respectively.   Here we use them to prove Theorem \ref{thm rigidity 2}.
\begin{proof}[Proof of Theorem \ref{thm rigidity 2}]
Take  $R_\ast:=2\max\{R_1,R_2\}$, where $R_1$, $R_2$ are the constants in Propositions \ref{pro2} and \ref{pro1} respectively. 

Since $w$ satisfy \eqref{SC1} and \eqref{rigidity condition1},  standard parabolic regularity theory (see also Lemma \ref{lem gradient estimate for self-similar solutions}) implies that if $\varepsilon_2$ is small enough, then $w$ is smooth in $B_{R_\ast}(0)$ and it satisfies
\[\text{dist}_{R_\ast-1}(w,\kappa)+\text{dist}_{R_\ast-1}\left(\Lambda(w), \frac{2}{p-1}\kappa\right)\leq \frac{\kappa}{10(p+1)}.\]
Moreover, since $R_{\ast}$ depends only on $n, p$ and $M$, $\varepsilon_2$ can be chosen to be a universal positive constant.
 Notice that Proposition \ref{pro1} extends
the scale  by a factor greater than one while Proposition \ref{pro2} only forces one to come in by a constant amount to get the improvement. Therefore, as long as the initial scale $R_{\ast}$ is large enough, we can iterate this process to obtain
\[\sup_{y\in\mathbb{R}^{n}}|w(y)-\kappa|+\sup_{y\in\mathbb{R}^{n}}\left|\Lambda(w)(y)-\frac{2}{p-1}\kappa\right|\leq \frac{\kappa}{10(p+1)}.\]
This yields $w$ is a bounded positive solution of \eqref{SC1} such that $\Lambda(w)$ does not change sign. By Proposition \ref{ChCS}, $w\equiv\kappa$.
\end{proof}

Now we turn to the proof of Theorem \ref{thm rigidity 1}. We need the following result to change from $L^2_w$ distance to $\text{dist}_R$ so that we can apply Theorem \ref{thm rigidity 2}.

\begin{lem}\label{lem L2 convergence lift to smooth}
Let $\{u_{k}\}\subset\mathcal{F}_{M}$ be a sequence of solutions of \eqref{eqn} such that $u_{k}\to u_{\infty}$ as $k\to\infty$. If  $u_{\infty}$ is bounded in $Q_{4}^{-}(0, -1)$, then  there exist two positive constants  $s_{\ast}$ $C_\ast$ depending on $n, M, p, \|u_{\infty}\|_{L^{\infty}(Q_{4}^{-}(0, -1))}$, and  an integer $k_{0}$ such that for any $k\geq k_{0}$,   
\begin{equation}\label{5.3}
|u_{k}|\leq C_\ast \quad \text{in} ~ \overline{B_2}\times [-1-s_\ast,-1+s_\ast].
\end{equation}
Furthermore, $u_k\to u_\infty$
 in $C^{2,1}(B_2\times (-1,-1+s_\ast))$.
\end{lem}
The proof of this lemma is similar to the one for Corollary \ref{coro convergence of singular points} and it is a direct consequence  of the $\varepsilon$-regularity (Proposition \ref{epsilonregularity}) and Schauder estimates for heat equation. However, because we need the quantitative estimate \eqref{5.3} in several places below, for completeness we will give a detailed  proof in Appendix \ref{sec proof of Lemma 4.6}. 
\begin{coro}\label{cor L2 convergence lift to smooth}
Let $\{u_{k}\}\subset\mathcal{F}_{M}$ be a sequence of solutions to \eqref{eqn} such that $u_{k}\to u_{\infty}$ as $k\to\infty$. If there exists a domain $\mathcal{D}\subset\R^n\times(-\infty,0)$ such that   $u_{\infty}\in C^{2,1}(\mathcal{D})$, then for any  $\mathcal{D}^\prime \Subset \mathcal{D}$, there exists an integer $k_{0}$ such that for any $k\geq k_{0}$,   $u_{k}\in C^{2,1}(\mathcal{D}^\prime)$ and $u_k\to u_\infty$
 in $C^{2,1}(\mathcal{D}^\prime)$.
\end{coro}
\begin{proof}
This can be obtained by combining Lemma \ref{lem L2 convergence lift to smooth} and a standard covering argument.
\end{proof}
\begin{lem}\label{prou-lem1}
Suppose $\{w_{k}\}\subset\mathcal{S}_{M}$  satisfies
\[\lim_{k\to\infty}\|w_{k}-\kappa\|_{L^2_w(\R^n)}=0.\]
Then $w_{k}\to\kappa$ in $C^2_{loc}(\mathbb{R}^{n})$.
\end{lem}
\begin{proof}
Let $u_{k}(x, t)=(-t)^{-\frac{1}{p-1}}w_{k}(\frac{x}{\sqrt{-t}})$. By Proposition \ref{prop compactness}, there exists a function $u_{\infty}\in\mathcal{F}_{M}$ such that $u_{k}\to u_{\infty}$. By the assumption,   $u_{k}\to \kappa(-t)^{-\frac{1}{p-1}}$ in $L_{loc}^{2}(\mathbb{R}^{n}\times(-\infty,0))$.
 The uniqueness of the strong convergence limit  implies $u_{\infty}=\kappa(-t)^{-\frac{1}{p-1}}$. By Corollary \ref{cor L2 convergence lift to smooth},  
$u_{k}\to\kappa(-t)^{-\frac{1}{p-1}}$ in $C_{loc}^{2,1}(\mathbb{R}^{n}\times (-\infty, 0))$. Restricting this convergence to  the $t=-1$ slice, we see that $w_{k}\to\kappa$ in $C_{loc}^{2}(\mathbb{R}^{n})$.
\end{proof}
By this lemma and a contradiction by compactness argument, we can find a universal positive constant $\varepsilon_1>0$ such that if $w\in\mathcal{S}_M$ satisfies $\|w-\kappa\|_{L^2_w(\R^n)}\leq \varepsilon_1$, then $\text{dist}_{R_\ast}(w,\kappa)\leq \varepsilon_2$. An application of Theorem \ref{thm rigidity 2} shows that $w\equiv\kappa$. This completes the proof of 
Theorem \ref{thm rigidity 1}.

\section{Uniqueness of the blow-down limit}
	\setcounter{equation}{0}
 
In order to prove Theorem \ref{maintheorem1}, one key step is to obtain the \emph{unique} asymptotic behavior of $u$ at large scales. This is the content of the following
\begin{prop}[Uniqueness of blow-down limit]\label{prouniquenessofblowdoen}
Assume $u$ satisfies the assumptions in Theorem \ref{maintheorem1}, then
\begin{equation}\label{prouniquenessofblowdoen1}
\lim_{t\to-\infty}(-t)^{\frac{1}{p-1}}u((-t)^{\frac{1}{2}}y, t)=\kappa \quad \text{in}~ C_{loc}^2(\R^n).
\end{equation}
\end{prop}
\begin{rmk}\label{pro uniquenessofblowup}
The same arguments can also be used to show the uniqueness  of the tangent functions, i.e. Item \ref{item:uniqueness} in Theorem \ref{maintheorem2}.
\end{rmk}

Roughly speaking, the above uniqueness result follows from two facts: (i) Theorem \ref{thm rigidity 1}, that is,  $\kappa$ is isolated in $\mathcal{S}_M$ with respect to the $L_{w}^{2}(\mathbb{R}^{n})$ distance; (ii) the set of tangent functions is path connected.

First, we give a preliminary result about the convergence  in self-similar coordinates.
\begin{prop}\label{prop convergence for self-similar equations}
Let $u\in\mathcal{F}_M$ and $w$ denote its self-similar transform. Then for any $s_k\to-\infty$ or $s_k\to+\infty$, there exists a subsequence of $s_k$ such that $w(y, s+s_k)$ converges to a limit $w_\infty$ strongly in $L^{p+1}_{loc}(\R^n\times\R)\cap H^1_{loc}(\R^n\times\R)$.  Moreover, $w_\infty\in\mathcal{S}_M$ and $w(\cdot, s_{k})$ converges to $w_{\infty}$ strongly in $L^2_w(\R^n)$.
\end{prop}
\begin{proof}
We need only to prove $w(\cdot, s_{k})$ converges to $w_{\infty}$ strongly in $L^2_w(\R^n)$. The other properties are just a rescaling of the convergence in Proposition \ref{prop blow-down limit}.

By Lemma \ref{Monotonicity formula} and Corollary \ref{coro 3.7}, the energy $E(w(\cdot,s))$ (see \eqref{energy} for definition) is a decreasing function of $s$ and
\begin{equation}\label{energy upper bound}
\lim_{s\to-\infty} E(w(\cdot, s))<+\infty.
\end{equation}
By the localized energy inequality \eqref{energy inequality I} and an approximation in $H^1_w(\R^n)\cap L^{p+1}_w(\R^n)$, we get Giga-Kohn's monotonicity formula in self-similar form,
\begin{equation}\label{energy idenity}
\frac{d}{ds}E(w(\cdot, s))=-\int_{\R^n} |\partial_sw(y,s)|^2 \rho(y)dy.
\end{equation}
We emphasize that this should be understood in the distributional sense, because the localized energy inequality \eqref{energy inequality I} holds only in the distributional sense.

Combining \eqref{energy idenity} with \eqref{energy upper bound}, we get
\[\int_{-\infty}^0\int_{\R^n} |\partial_sw(y,s)|^2 \rho(y)dyds<+\infty.\]
On the other hand, testing \eqref{self-similar eqn} with $w$, we get (still understood in the distributional sense)
\[ \int_{\R^n} \partial_sw w\rho dy= \int_{\R^n}\left[|w|^{p+1}-|\nabla w|^2-\frac{1}{p-1}w^2\right] \rho dy.\]
Then following the proof of \cite[Proposition 2.2]{Giga-Kohn2}, we get
\[ \sup_{s<0} \int_{s-1}^s \int_{\R^n} \left[|\partial_sw(y,s)|^2+|\nabla w(y,s)|^2+|w(y,s)|^{p+1} \right]\rho(y)dy<+\infty.\]
By Sobolev embedding theorem, this implies that $w(\cdot, s)\in C^{\alpha}((-\infty,0];L^2_w(\R^n))$ for some $\alpha>0$. It follows that $w(\cdot, s)-w_{\infty} \in C^{\alpha}((-\infty,0];L^2_w(\R^n))$.

Assume  $w(\cdot, s_{k})$ does not converge strongly in $L^2_w(\R^n)$ to $w_{\infty}$. Then there exists a constant $\delta_{0}$ such that 
\[\limsup_{k\to\infty}\int_{\mathbb{R}^{n}}|w(\cdot, s_{k})-w_{\infty}|^{2}\rho dy>\delta_{0}.\]
Since $w(\cdot, s)-w_{\infty} \in C^{\alpha}((-\infty,0];L^2_w(\R^n))$,  there exist a constant $0<\delta_{1}<1/16$ and a subsequence $\{k_{l}\}$ such that for each $l$,
\[\int_{\mathbb{R}^{n}}|w(\cdot, s)-w_{\infty}|^{2}\rho dy>\frac{\delta_{0}}{2},\quad\text{for}~s\in (s_{k_{l}}-\delta_{1}, s_{k_{l}}+\delta_{1}).\]
In particular, we have
\begin{equation}\label{spacetimeestimate}
\int_{s_{k_{l}}-\delta_{1}}^{s_{k_{l}}+\delta_{1}}\int_{\mathbb{R}^{n}}|w(\cdot, s)-w_{\infty}|^{2}\rho dyds>\delta_{0}\delta_{1}.
\end{equation}

By Lemma \ref{control outside region}, there  exists a positive constant $R$ such that for each $k_{l}$,
\begin{equation}\label{spacetimeestimate1}
\int_{s_{k_{l}}-\delta_{1}}^{s_{k_{l}}+\delta_{1}}\int_{\mathbb{R}^{n}\backslash B_{R}}|w(\cdot, s)-w_{\infty}|^{2}\rho dyds<\frac{\delta_{0}\delta_{1}}{2}.
\end{equation}
Since $w(y,s+s_k)$ converges to  $w_\infty$ strongly in $L^{p+1}_{loc}(\R^n\times\R)$, there exists a constant $l_{0}$ such that for each $l>l_{0}$,
\begin{equation}\label{spacetimeestimate2}
\int_{s_{k_{l}}-\delta_{1}}^{s_{k_{l}}+\delta_{1}}\int_{ B_{R}}|w(\cdot, s)-w_{\infty}|^{2}\rho dyds<\frac{\delta_{0}\delta_{1}}{2}.
\end{equation}
By \eqref{spacetimeestimate1} and \eqref{spacetimeestimate1}, we know that for each $l>l_{0}$,
\[
\int_{s_{k_{l}}-\delta_{1}}^{s_{k_{l}}+\delta_{1}}\int_{\mathbb{R}^{n}}|w(\cdot, s)-w_{\infty}|^{2}\rho dyds<\delta_{0}\delta_{1},
\]
which contradicts \eqref{spacetimeestimate}.
\end{proof}
\begin{rmk}\label{rmk continuity of blow-up sequence in L2}
In the above proof, we have also showed that $w(\cdot, s)$ is continuous in $s$ with respect to the $L^2_w(\R^n)$ distance. 

Furthermore, because self-similar transform is equivalent to the blow-up procedure, for any base point $(x_0,t_0)$, the blow-up (or blow-down) sequence  $u_{x_0,t_0,\lambda}$ is also   continuous in $\lambda$ with respect to the $L^2_w(\R^n)$ distance.
\end{rmk} 

Now let us prove Proposition \ref{prouniquenessofblowdoen}.
\begin{proof}[Proof of Proposition \ref{prouniquenessofblowdoen}] 
Let $w$ be the self-similar transform of $u$.

\textbf{Step 1.} By Proposition \ref{prop convergence for self-similar equations}, $w\in C^{\alpha}((-\infty,0], L^2_w(\R^n))$, where $0<\alpha<1$ is a positive constant. As a consequence, 
\[ d(s):=\|w(s)-\kappa\|_{L^2_w(\R^n)}\]
is a bounded, continuous function on $(-\infty,0]$.

\textbf{Step 2.} There exists a sequence $s_{k,1}\to-\infty$ such that $\liminf_{k\to\infty}d(s_{k,1})=0$.

Let $\{\lambda_{k}\}$ be a sequence satisfying the assumptions in Theorem \ref{maintheorem1}. For each $k$,  set $s_{k,1}=-\log \lambda_k$.
  Then by the assumptions in Theorem \ref{maintheorem1}, we   deduce  that $w_{k}:=w(\cdot, s+s_{k,1})\to \kappa$ in $L^2_w(\R^n)$. Hence
  \[ d(s_{k,1})=\|w(\cdot, s_{k,1})-\kappa\|_{L^2_w(\R^n)}\to 0.\]

\textbf{Step 3.}  $\lim_{s\to-\infty}d(s)=0$. 

Assume by the contrary, there exists a sequence $s_{k,2}\to-\infty$ such that $d(s_{k,2})\not\to 0$. By the continuity of $d$ and the result proved in Step 2, there exists a new sequence $s_{k,3}$ satisfying
\[  \lim_{k\to\infty}d(s_{k,3})\in(0, \varepsilon_1],\]
where $\varepsilon_1$ is the constant in Theorem \ref{thm rigidity 1}.

Let $w_{k}(y, s)=w(y, s+s_{k,3})$. By Proposition \ref{prop convergence for self-similar equations}, there exists a function $w_\infty\in \mathcal{S}_{M}$ such that
$w_{k}(\cdot, 0)\to w_\infty$ strongly in $L^2_w(\R^n)$. Then
\[ \|w_\infty-\kappa\|_{L^2_w(\R^n)}=\lim_{k\to\infty}d(s_{k,3})\in(0,\varepsilon_1].\]
This is a contradiction with Theorem \ref{thm rigidity 1}.

\textbf{Step 4:} $\lim_{s\to-\infty}w(y, s)=\kappa$ in $C^2_{loc}(\R^n)$.

The result in Step 3 says that $w(y, s)\to\kappa$ in $L^2_w(\R^n)$ as $s\to-\infty$. Then by Corollary \ref{cor L2 convergence lift to smooth}, we can lift this  to $C^2_{loc}(\R^n)$ convergence.
\end{proof}

\section{An auxiliary function}
\setcounter{equation}{0}

In this section, $u$ always denotes a solution of \eqref{eqn}, which satisfies the assumptions in Theorem \ref{maintheorem1}.

\begin{defi}\label{definefunctionT}
  For any $x\in\mathbb{R}^{n}$, let $RT(x)$ be the infimum of $s$ satisfying the following conditions:
\begin{itemize}
    \item[(1)] $\{x\}\times (-\infty,-s)\subset \text{Reg}(u)$;
    \item[(2)] $u>0$ on $\{x\}\times (-\infty,-s)$;
    \item[(3)] $|\Delta u(x, t)|\leq \frac{1}{2} u(x, t)^{p}$ for any $t<-s$.
\end{itemize}
\end{defi}
As a consequence of Proposition \ref{prouniquenessofblowdoen}, the function $RT$ is well defined.
 The function $RT$  introduced above satisfies the following two properties.
 \begin{lem}\label{Tproperties1}
 Let  $RT$ be 
 the function defined in Definition \ref{definefunctionT}, then
 \begin{itemize}
     \item[(1)] for any $x\in\R^n$, $RT(x)<+\infty$;
     \item[(2)] $RT(x)=o(|x|^{2})$ as $|x|\to+\infty$.
 \end{itemize}
 \end{lem}
 \begin{proof}
 (1) By Proposition \ref{prouniquenessofblowdoen}, there exists a large constant $C$  such that in $\{t<-C|x|^2-C\}$,
\[ \left|u(x,t)-\kappa(-t)^{-\frac{1}{p-1}}\right|\leq \frac{1}{C}(-t)^{-\frac{1}{p-1}}.\]
Then by standard parabolic estimates, we deduce that in $\{t<-2C|x|^2-2C\}$, $u$ is smooth and positive, and there holds
\[ |\Delta u(x,t)|\leq \frac{1}{C} (-t)^{-\frac{p}{p-1}}<\frac{1}{2} u(x,t)^{p}.\]
Hence $RT(x)\leq 2C|x|^2+2C$.

 (2) Assume by the contrary that there exists an $\varepsilon>0$ and a sequence $\{x_{k}\}$ such that $\lim_{k\to+\infty}|x_{k}|=+\infty$ and
 \begin{equation}\label{contradictionassumption}
 RT(x_{k})\geq \varepsilon|x_{k}|^{2}.\tag{7.1}
 \end{equation}
 Set $\lambda_{k}=RT(x_{k})^{\frac{1}{2}}$, which goes to infinity as $k\to\infty$.
  By \eqref{contradictionassumption}, we have 
  \begin{equation}\label{contradictionassumptioneq}
  \frac{|x_{k}|}{\lambda_{k}}\leq \varepsilon^{-\frac{1}{2}}.\tag{7.2}
  \end{equation}
  
  By the definition of $RT$, there are three possibilities: 
  \begin{itemize}
      \item[(1)]   $(x_{k}, -RT(x_{k}))$ is a singular point;
      \item[(2)]  $(x_{k}, -RT(x_{k}))$ is a regular point but $u(x_{k}, -RT(x_{k}))=0$; 
      \item[(3)]    $(x_{k}, -RT(x_{k}))$ is a regular point, $u(x_{k}, -RT(x_{k}))>0$ and
  \[|\Delta u(x_{k}, -RT(x_{k}))|=\frac{1}{2} u(x_{k}, -RT(x_{k}))^p.\]
  \end{itemize}
 
  Define
 \[u_{k}(x, t)=\lambda_{k}^{\frac{2}{p-1}}u(\lambda_{k}x, \lambda_{k}^{2}t).\]
 It is still an ancient solution of \eqref{eqn}. Moreover, we have the following three possibilities:
 \begin{itemize}
     \item[(1')] $(x_{k}/\lambda_{k}, -1)$ is a singular point of $u_k$;
     \item[(2')] $(x_{k}/\lambda_{k}, -1)$ is a regular point of $u_k$ but $u_{k}(x_{k}/\lambda_{k}, -1)=0$; 
     \item[(3')] $(x_{k}/\lambda_{k}, -1)$ is a regular point of $u_k$, $u_{k}(x_{k}/\lambda_{k}, -1)>0$
 and
 \[\left|\Delta u_{k}\left(\frac{x_{k}}{\lambda_{k}}, -1\right)\right|=\frac{1}{2} u_{k}\left(\frac{x_{k}}{\lambda_{k}}, -1\right)^p.\]
  \end{itemize}
  Since $\lambda_{k}$ goes to infinity as $k\to\infty$,  we get from  Proposition \ref{prouniquenessofblowdoen} that
 \begin{equation}\label{6.6'}
 u_{k}(x, t)\to \kappa(-t)^{-\frac{1}{p-1}},\quad\text{as}~k\to\infty.\tag{7.3}
 \end{equation}
 Combining \eqref{contradictionassumptioneq}, \eqref{6.6'}  with
  Corollary \ref {cor L2 convergence lift to smooth}, we deduce that $u_{k}$ is smooth in a fixed neighborhood of $ (x_{k}/\lambda_{k}, -1)$ provided $k$ is large enough, and the convergence in \eqref{6.6'} can be lifted to $C^{2,1}$ convergence in this neighborhood.  This excludes the above Case (1'). 
  
  Proposition \ref{prouniquenessofblowdoen} implies that the convergence in \eqref{6.6'} is
 uniform on any compact subset, so  Case (2') is also  excluded. Thus only Case (3') is possible. In other words,  $u_{k}(x_{k}/\lambda_{k}, -1)>0$ provided that $k$ is large enough.  Then
 \begin{equation}\label{endtimecondition}
 \left|\Delta u_{k}(\frac{x_{k}}{\lambda_{k}}, -1)\right|=\frac{1}{2} u_{k}
 \left(\frac{x_{k}}{\lambda_{k}}, -1
 \right)^{p}\to \frac{1}{2}\kappa^p.\tag{7.4}
 \end{equation}
But using  the smooth convergence of $u_k$ again, we also have
 \[\Delta u_{k}\left(\frac{x_{k}}{\lambda_{k}}, -1\right)\to 0\]
 as $k\to\infty$, which contradicts \eqref{endtimecondition}.
 \end{proof}
 \begin{rmk}
By a slight modification of the proof of item (2) in Lemma  \ref{Tproperties1}, we can also obtain that $u(x, 0)=o(|x|^{-\frac{2}{p-1}})$ as $|x|\to\infty$. Similar estimate has been obtained by Matano and Merle \cite[Remark 2.6]{Merle-Matano2011} under the assumption that solutions to \eqref{eqn} are radially symmetric.
 \end{rmk}

  \section{A conditional Liouville theorem}
  	\setcounter{equation}{0}
   
 In this section, our main objective is to prove the following conditional Liouville  type result. 
 \begin{prop}\label{preproposition}
 Let $u$ be a positive smooth solution of \eqref{eqn} in $\mathbb{R}^{n}\times (-\infty, -T_{0})$, where $T_{0}\geq 0$ is a nonnegative constant.
Assume there exists a  constant $0<\varepsilon<1$  such that 
 \begin{equation}\label{additional assumption}
     |\Delta u(x, t)|\leq \varepsilon u(x, t)^p,\quad\text{in}~\mathbb{R}^{n}\times (-\infty, -T_{0}).
 \end{equation}
Then  there exist a constant $T\geq -T_{0}$ such that
\[u(x, t)\equiv \kappa(T-t)^{-\frac{1}{p-1}},\quad\text{in}~\mathbb{R}^{n}\times (-\infty, -T_{0}).\]
 \end{prop}
  Proposition \ref{preproposition} is almost the same with Theorem \ref{maintheorem1}, except that we add an additional assumption \eqref{additional assumption}.

To prove this proposition, we need the following lemma.
 \begin{lem}\label{lem2024-0416}
Under the assumptions of Proposition \ref{preproposition},
 we have
 \begin{equation}\label{lem2024-0416'}
\lim_{t\to-\infty}(-T_{0}-t)^{\frac{1}{p-1}}u((-T_{0}-t)^{\frac{1}{2}}y, t)=\kappa \quad\text{in}~~ C^\infty_{loc}(\R^n).
\end{equation}
 \end{lem}
 \begin{proof}
Combining \eqref{additional assumption} with \eqref{eqn}, we see
  \begin{equation}\label{maindifferentialinequality'}
 \partial_{t}u\geq (1-\varepsilon)u^p,\quad\text{in}~\mathbb{R}^{n}\times (-\infty, -T_{0}).
 \end{equation}
 This inequality is equivalent to
 \begin{equation}\label{utimeinequality}
 \partial_{t}u^{1-p}\leq -(p-1)(1-\varepsilon).
 \end{equation}
Fix a small positive constant $\delta$  so that $-T_{0}-\delta>t$. Integrating \eqref{utimeinequality} from $(x, t)$ to $(x, -T_{0}-\delta)$ in time, we get
 \[u(x, -T_{0}-\delta)^{1-p}-u(x, t)^{1-p}\leq -(p-1)(1-\varepsilon)(-T_{0}-\delta-t).\]
After letting $\delta\to0$, we obtain
 \begin{equation}\label{ODE upper bound}
     u(x, t)\leq M(-T_{0}-t)^{-\frac{1}{p-1}},\quad\text{in}~\mathbb{R}^{n}\times (-\infty, -T_{0})
 \end{equation}
 with $M=((p-1)(1-\varepsilon))^{-\frac{1}{p-1}}$.
 
 Consider the self-similar transform of $u$ with respcet to $(0,-T_0)$, that is, 
 \begin{equation}\label{newselfsimilartransform}
 w(y, s)=(-T_{0}-t)^{\frac{1}{p-1}}u(x, t),\quad x=(-T_{0}-t)^{\frac{1}{2}}y,\quad s=-\log(-T_{0}-t).
 \end{equation}
 Then $w$ is a solution of \eqref{self-similar eqn} satisfying
 \begin{equation}\label{6.6}
     0\leq w\leq M,\quad\text{in}~~\mathbb{R}^{n}\times (-\infty, +\infty).
 \end{equation}

 By \eqref{newselfsimilartransform}, we have
\[\partial_{t}u=(-T_{0}-t)^{-\frac{1}{p-1}-1}\left(\frac{1}{p-1}w+\frac{y}{2}\cdot\nabla w+\partial_{s}w\right)\]
and
\[u^p=(-T_{0}-t)^{-\frac{1}{p-1}-1}w^p.\]
Plugging these two relations into \eqref{maindifferentialinequality'}, we get
\[\frac{1}{p-1}w+\frac{y}{2}\cdot\nabla w+\partial_{s}w\geq (1-\varepsilon)w^p,\quad\text{in}~\mathbb{R}^{n}\times (-\infty, +\infty).\]
 
Let $\{s_{k}\}$ be a sequence such that $\lim_{k\to-\infty}s_{k}=-\infty$ and set $w_{k}(y, s)=w(y, s+s_{k})$. By Proposition \ref{prop convergence for self-similar equations} and \eqref{6.6},  there exists a function $w_{\infty}$ such that $w_{k}\to w_{\infty}$ in  $C^\infty_{loc}(\mathbb{R}^{n}\times(-\infty,\infty))$. Moreover, $w_{\infty}$ is independent of $s$, so by \eqref{6.6}, it   is a bounded solution of \eqref{SC1}. As a consequence, $\partial_sw_{k}\to 0$ in  $C^\infty_{loc}(\mathbb{R}^{n}\times(-\infty,\infty))$.

Because $w_k$ is just a translation (in time) of $w$, it also satisfies
\[\frac{1}{p-1}w_{k}+\frac{y}{2}\cdot\nabla w_{k}+\partial_{s}w_{k}\geq (1-\varepsilon)w_{k}^{p},\quad\text{in}~\mathbb{R}^{n}\times (-\infty, +\infty).\]
By the above convergence of $w_k$, taking limit in $k$ gives
\[\frac{1}{p-1}w_{\infty}+\frac{y}{2}\cdot\nabla w_{\infty}\geq (1-\varepsilon)w_{\infty}^{p},\quad\text{in}~\mathbb{R}^{n}.\]
Hence $w_{\infty}$ is a bounded solution of \eqref{SC1} satisfying
\[\frac{1}{p-1}w_{\infty}+\frac{y}{2}\cdot\nabla w_{\infty}\geq 0,\quad\text{in}~\mathbb{R}^{n}.\]
By Proposition \ref{ChCS}, $w_{\infty}\equiv 0$ or $\kappa$. Because we have assumed that $u$ is positive,   \cite[Proposition 2.1]{Merle-Zaag2000} implies for any $s\in(-\infty, +\infty)$, the   energy $E(w(\cdot, s))$ (as defined in \eqref{energy}) is positive. If $w_{\infty}=0$, this will contradict \cite[Proposition 3]{Giga-Kohn1985}. Hence it is only possible  that $w_{\infty}=\kappa$.

Since $\{s_{k}\}$ can be arbitrary, we conclude that $\lim_{s\to-\infty}w(y, s)=\kappa$, where the convergence is in $C^\infty_{loc}(\R^n)$. Then by the definition of $w$, we get \eqref{lem2024-0416'}.
 \end{proof}
 \begin{rmk}
 Consider the Cauchy problem for \eqref{eqn} with initial value 
$u(\cdot, 0)=\kappa+\psi$. Let $C^{2}(\mathbb{R}^n)$ be the space of $C^{2}$ continuous functions, equipped with the norm
\[\|\psi\|_{C^{2}(\mathbb{R}^{n})}:= \sup_{x\in\mathbb{R}^{n}}[|\psi(x)|+|\nabla \psi(x)|+|\nabla^{2}\psi(x)|].\]
If $\|\psi\|_{C^{2}(\mathbb{R}^{n})}$ is small enough, then there exists $\varepsilon>0$ such that
\[|\Delta\psi|\leq \varepsilon (\kappa+\psi)^{p},\quad\text{in}~\mathbb{R}^{n}.\]
If $u$ blows up at time $T$, then the maximum principle implies
 \[\partial_{t}u\geq (1-\varepsilon)u^p,\quad\text{in}~\mathbb{R}^{n}\times(0, T).\]
 Similar to the proof of Lemma \ref{lem2024-0416}, one can show that if $(x_{0}, T)\in\text{Sing}(u)$, then $\mathcal{T}(x_{0}, T; u)=\{\kappa\}$. This implies the ODE blow-up is stable under small perturbations. The stability of the ODE blow-up has been studied intensively in the  literature, see for instance \cite{Fermanian-Merle-Zaag,Merle-Zaag1997}. In a recent paper \cite{Harada}, Harada generalized the results in \cite{Fermanian-Merle-Zaag} to the case that the exponent $p$ is supercritical. In particular, he proved that nondegenerate ODE type blow-up \footnote{The definition of nondegenerate ODE type blow-up can be found in \cite{Harada}} is stable (under small perturbations in
the $L^{\infty}$ norm on the initial data). This is different from our stability result.
 \end{rmk}
 \begin{prop}\label{calssifyingboundednonegativesolution}
 Let $w$ be a bounded, positive solution of the self-similar equation
 \eqref{self-similar eqn} in $\R^n\times(-\infty,+\infty)$.
If
 \begin{equation}\label{7.1}
 \lim_{\tau\to-\infty}w(y, \tau)=\kappa \quad \text{in} ~~ C_{loc}^{\infty}(\R^n),
 \end{equation}
 then 
 \begin{itemize}
 \item[(i)] either $w\equiv \kappa$; 

 \item[(ii)] or there exists $s_{0}\in\mathbb{R}$ such that
  \[w(y, s)=\kappa(1+e^{s-s_0})^{-\frac{1}{p-1}},\quad\text{in}~\mathbb{R}^{n}\times (-\infty, +\infty).\]
 \end{itemize}
 \end{prop}
 \begin{proof}
 Let $\{s_{j}\}$ be a sequence tending to $+\infty$ and let $w_{j}(y, s)=w(y, s+s_{j})$. By Proposition \ref{prop convergence for self-similar equations}, $\{w_{j}\}$ converges to some $w_{\infty}\in\mathcal{S}_M$ in the sense specified there. Because $w\in L^\infty(\R^n\times(-\infty,+\infty))$, by standard parabolic regularity theory and Arzela-Ascoli theorem, the convergence also holds in $C^\infty_{loc}(\R^n\times(-\infty,+\infty))$.  This implies $w_{\infty}$ is a bounded nonnegative solution of \eqref{SC1}. By the monotonicity of $E(w(\cdot, s))$ (see \eqref{energy idenity} or \cite[Proposition 1']{Giga-Kohn1985}), the energy
 $E(w_{\infty})$ is independent of the choice of the sequence $\{s_{j}\}$, where $E(w_{\infty})$ is the energy of $w_{\infty}$ defined in \eqref{energy}.  Moreover, by the monotonicity of $E(w(\cdot, s))$ and \eqref{7.1},
 \begin{equation}\label{8.9}
 E(w_{\infty})\leq E(\kappa).
 \end{equation}
 On the other hand, Theorem \ref{ME} implies if $w_{\infty}\neq 0, \pm\kappa$, then
 \[E(w_{\infty})>E(\kappa),\]
 which is a contradiction with \eqref{8.9}. 
 Thus  either $w_{\infty}=0$ or $w_{\infty}=\kappa$. 
 
 If $w_{\infty}=\kappa$, then \eqref{energy idenity} (or \cite[Proposition 3]{Giga-Kohn1985})  implies $w\equiv \kappa$. This is exactly (i).
 
 If $w_{\infty}=0$, then  $\lim_{\tau\to+\infty}w(y, \tau)=0$. In this case, it has essentially been proved by Merle and Zaag \cite[Theorem 1.4]{Merle-Zaag} that (ii)  holds. As mentioned in the introduction, in the statement of \cite[Theorem 1.4]{Merle-Zaag}, they assumed that the exponent $p$ satisfies $1<p<\frac{n+2}{n-2}$. However, this condition is only used to prove that if $w\neq\kappa$, then
 \[\lim_{\tau\to -\infty}w(y, \tau)=\kappa ~~\text{or}~~0.\]
 Since we have verified that these two conditions hold in our setting,
 we can repeat their arguments without any modification to conclude the proof.
 \end{proof}
 \begin{coro}\label{classifyancientsolution}
Let $u$ be a positive smooth solution of the equation \eqref{eqn} in $\mathbb{R}^{n}\times (-\infty, -T_{0})$, where $T_{0}\geq 0$. If
\begin{equation}
\lim_{t\to-\infty}(-t)^{\frac{1}{p-1}}u((-t)^{\frac{1}{2}}y, t)=\kappa
\end{equation}
and there exists a positive constant $M$ such that
\begin{equation}\label{lowerbound}
(-T_{0}-t)^{\frac{1}{p-1}}u(x, t)\leq M, \quad\text{in}~\mathbb{R}^{n}\times (-\infty, -T_{0}),
\end{equation}
then there exists $T\in\mathbb{R}$ such that
\begin{equation}\label{ODEsolution}
u(x, t)=\kappa(T-t)^{-\frac{1}{p-1}},\quad\text{in}~\mathbb{R}^{n}\times(-\infty, -T_{0}).
\end{equation}
\end{coro}
\begin{proof}
Let $w$ be the self-similar transform of $u$ as defined in \eqref{newselfsimilartransform}. It satisfies the conditions in Proposition \ref{calssifyingboundednonegativesolution}. Hence \eqref{ODEsolution} follows directly from an application of Proposition \ref{calssifyingboundednonegativesolution}.
\end{proof}
The above results are all concerned with the description of solutions backward in time. We also need a result on the continuation of solutions forward in time. This is  a uniqueness result for Cauchy problem to \eqref{eqn} in the setting of suitable weak solutions, but only for a very special class of such solutions.
\begin{lem}\label{lem uniqueness for Cauchy problem}
   Suppose $u\in\mathcal{F}_M$, and there exist $T_0\geq 0, T\in\mathbb{R}$ such that in $\R^n\times(-\infty,-T_0)$, 
   \begin{equation}\label{ODE solution as initial value}
   u(x,t)\equiv \kappa(T-t)^{-\frac{1}{p-1}}.
   \end{equation}
   Then $T\geq 0$ and \eqref{ODE solution as initial value} holds in the entire $\R^n\times(-\infty,0)$.
\end{lem}
\begin{proof}
First, if $T<-T_{0}$, then 
\[\int_{Q^{-}_{1}(0, T)}|u|^{p+1}dxdt=\kappa^{p+1}\int_{Q_{1}^{-}}(-t)^{\frac{p+1}{p-1}}dxdt=+\infty.\]
This contradicts our assumption that $u\in\mathcal{F}_{M}$. Therefore,  $T\geq -T_{0}$. In particular,   if $T_{0}=0$, then $T\geq 0$ and the proof is complete.

If $T_{0}>0$, the above analysis also imply $T> -T_{0}$. This yields $u$ is uniformly bounded on $\mathbb{R}^{n}\times (-\infty, -T_{0})$.
   By Lemma \ref{lem L2 convergence lift to smooth}, we can find a positive constant $s_{\ast}$ such that for any $(x,t)\in\R^n\times[-T_0,T_0+s_{\ast}]$,
    $u$ is smooth and bounded in $\R^n\times[-T_0,-T_0+s_{\ast}]$. Because solution to the Cauchy problem for \eqref{eqn} is unique  in $L^\infty(\R^n)$, \eqref{ODE solution as initial value} holds up to $-T_0+s_{\ast}$. Then by a continuation in time, we get the conclusion.
\end{proof}

 \begin{proof}[Proof of Proposition \ref{preproposition}]
  It is a direct consequence of Lemma \ref{lem2024-0416}, Corollary \ref{classifyancientsolution} and Lemma \ref{lem uniqueness for Cauchy problem}.
  \end{proof}

 \section{Proof of Theorem \ref{maintheorem1}}
 	\setcounter{equation}{0}

 In this section, we prove Theorem \ref{maintheorem1}. The proof is based on compactness and uniqueness of solutions to the Cauchy problem of \eqref{eqn}.

 \begin{proof}[Proof of Theorem \ref{maintheorem1}]
  Let $RT$ be the function defined by Definition \ref{definefunctionT}. 
  
  First assume   $RT$ is bounded, that is, there exists a positive constant $T_{0}$ such that
  \[RT(x)<T_{0},\quad\text{in}~\mathbb{R}^{n},\]
  By the definition of $RT$, 
 $u$ is smooth and positive in $\mathbb{R}^{n}\times (-\infty, -T_{0})$. Moreover, we have
 \[|\Delta u|\leq\frac{1}{2} u^{p},\quad\text{in}~\mathbb{R}^{n}\times (-\infty, -T_{0}).\]
  Thus we can apply Proposition \ref{preproposition} and Lemma \ref{lem uniqueness for Cauchy problem} to conclude the proof.

  In the rest of the proof, we may assume that $RT$ is unbounded.
  By Lemma \ref{Tproperties1},
  \begin{equation}\label{9.1}
      RT(x)=o(|x|^{2}), \quad\text{as}~|x|\rightarrow+\infty.
  \end{equation}
 Let
 \[f(r)=\sup_{x\in B_{r}(0)}RT(x),\]
 which is an increasing function of $r$ and satisfies $\lim_{r\to\infty}f(r)=\infty$.
 
 {\bf Claim.} There exists a sequence $\{r_{k}\}$ such that $\lim_{k\to\infty}r_{k}=\infty$ and  for each $k$, \[f(2r_{k})\leq 4f(r_{k}).\]

Assume this claim is false. Then there exists an $R$ such that for all $r\geq R$, 
\[f(2r)>4f(r).\]
Without loss of generality, assume $R=2^{k_0}$ for some $k_0\in\mathbb{N}$, so we have
  \[f(2^{k})>4f(2^{k-1}), \quad \forall k\geq k_0.\]
 An iteration of this inequality gives for each $k\geq k_{0}+1$,
 \[f(2^{k})>4^{k-k_{0}}f(2^{k_{0}}).\]
 For each $k\geq k_{0}+1$, let $x_{k}\in B_{2^{k}}(0)$ be a point such that
 $RT(x_{k})>f(2^{k})/4$, then
 \[RT(x_{k})>\frac{f(2^{k})}{4}>4^{-1-k_{0}}f(2^{k_{0}}) 4^{k} \geq 4^{-1-k_{0}}f(2^{k_{0}})|x_{k}|^{2}.\]
Letting $k\to\infty$, we see \[\limsup_{|x|\to\infty}\frac{RT(x)}{|x|^{2}}\geq 4^{-1-k_{0}}f(2^{k_{0}})>0.\]
This contradicts \eqref{9.1}. The claim is thus proved. 

 By this claim,   there exists a sequence $\{x_{k}\}$ satisfying $\lim_{k\rightarrow+\infty}|x_{k}|=+\infty$ and
 \[\sup\limits_{B_{R_{k}}(x_{k})}RT(x)\leq 4RT(x_{k}),\]
 where  $R_{k}=|x_{k}|$.  
 
 Denote $\lambda_{k}=RT(x_{k})^{\frac{1}{2}}$. Set
 \[u_{k}(x, t)=\lambda_{k}^{\frac{2}{p-1}}u(x_k+\lambda_{k}x, \lambda_{k}^{2}t)\]
and
 \[RT_{k}(x)=\lambda_{k}^{-2}RT(x_{k}+\lambda_{k}x).\]
 Then
 \begin{equation}\label{Tscaling}
 \sup\limits_{B_{R_{k}/\lambda_{k}}(0)}RT_{k}(x)\leq 4,\quad RT_{k}(0)=1.
 \end{equation}
 By Lemma \ref{Tproperties1}, we know that
 \[\frac{R_{k}}{\lambda_{k}}\rightarrow \infty,\quad\text{as}~k\rightarrow\infty.\]
By Proposition \ref{prop compactness}, $u_k$ converges to a limit $u_\infty$ in the sense specified there, where $u_\infty\in\mathcal{F}_{M}$ is an ancient solution of
\eqref{eqn}.

 On the other hand, as in the proof of Lemma \ref{lem2024-0416}, using \eqref{Tscaling} we can show that there exists a constant $M$ independent of $k$ such that
\[0< u_{k}\leq M(-4-t)^{-\frac{1}{p-1}}, \quad\text{in}~B_{R_{k}/\lambda_{k}}(0)
\times (-\infty, -4).\]
By standard parabolic estimates and   Arzela-Ascoli theorem,  $u_{k}\to u_{\infty}$ in $C^\infty_{loc}(\mathbb{R}^{n}\times (-\infty, -4))$. In particular, $u_{\infty}\in C^\infty(\mathbb{R}^{n}\times (-\infty, -4))$.

 Letting $k\to\infty$ in \eqref{Tscaling}, we obtain
 \begin{equation}\label{Tinftysacling}
 |\Delta u_{\infty}|\leq \frac{1}{2} u_{\infty}^{p},\quad\text{in}~\mathbb{R}^{n}\times (-\infty, -4).
 \end{equation}
 
In conclusion, $u_{\infty}$ is an ancient solution of \eqref{eqn}, and  in $\mathbb{R}^{n}\times(-\infty, -4)$, it is positive, smooth and  \eqref{Tinftysacling} holds. By Proposition \ref{preproposition},  $u_{\infty}$ depends only on $t$. Because $u_{\infty}$ is an ancient solution of \eqref{eqn}, by Lemma \ref{lem uniqueness for Cauchy problem}, there exists $T\geq 0$ such that
\begin{equation}\label{8.3}
u_{\infty}(x, t)=\kappa(T-t)^{-\frac{1}{p-1}} \quad \text{in} ~ \R^n\times(-\infty,0).
\end{equation}
In particular, $u_{\infty}$ is smooth on $\mathbb{R}^{n}\times(-\infty, 0)$. Then by  Corollary \ref{cor L2 convergence lift to smooth}, we deduce that $u_{k}\to u_{\infty}$ in $C^\infty_{loc}(\mathbb{R}^{n}\times (-\infty, 0))$. In particular, for all $k$  large enough,   $u_{k}$ is smooth in a fixed neighborhood of $(0, -1)$. Since 
$RT_{k}(0)=1$,  there are only two possibilities: either $u_{k}(0, -1)=0$ or
\[|\Delta u_{k}(0, -1)|=\frac{1}{2} u_{k}(0, -1)^{p}.\]

If there exists a subsequence $k_{j}$ such that $u_{k_{j}}(0, -1)=0$, then $u_{\infty}(0, -1)=0$, which is impossible by \eqref{8.3}. 

If there exists a subsequence $k_{j}$ such that 
\[|\Delta u_{k_{j}}(0, -1)|=\frac{1}{2} u_{k_{j}}(0, -1)^{p},\]
then passing to the limit we obtain
\[|\Delta u_{\infty}(0, -1)|=\frac{1}{2} u_{\infty}(0, -1)^{p}=\frac{1}{2}\kappa^p(T+1)^{-\frac{p}{p-1}}\neq 0.\]
On the other hand, \eqref{8.3} gives
\[|\Delta u_{\infty}(0, -1)|=0.\]
We arrive at a contradiction again.

Since we can get a contradiction in both cases, we conclude that a sequence of solutions of \eqref{eqn} satisfying \eqref{Tscaling} can not exist. This contradicts our assumption, so the proof of Theorem \ref{maintheorem1}  is  complete.
 \end{proof}
\begin{proof}[Proof of Corollary \ref{Low-entropy ancient solutions}]
Let $\{\lambda_{k}\}$ be a sequence such that $\lambda_{k}\to\infty$ as $k\to\infty$. For each $k$, we set
\[u_{k}(x, t)=\lambda_{k}^{\frac{2}{p-1}}u(\lambda_{k}x, \lambda_{k}^{2}t).\]
By Proposition \ref{prop blow-down limit}, there exists a function $w_{\infty}\in\mathcal{S}_{M}$ such that \[u_{k}(x, t)\to(-t)^{-\frac{1}{p-1}}w_{\infty}\left(\frac{x}{\sqrt{-t}}\right),\quad\text{as}~k\to\infty.\]
Since
\[\sup_{t\in (-\infty, 0)}\lambda(u(\cdot, t))\leq \left(\frac{1}{2}-\frac{1}{p+1}\right)\left(\frac{1}{p-1}\right)^{\frac{p+1}{p-1}},\]
then
\[E(w_{\infty})\leq \left(\frac{1}{2}-\frac{1}{p+1}\right)\left(\frac{1}{p-1}\right)^{\frac{p+1}{p-1}}.\]
By \cite[Theorem 1.7]{Wang-Wei-Wu} and its proof, and because $w_\infty$ is non-negative,  either $w_{\infty}=0$ or $w_{\infty}=\kappa$. 

If $w_{\infty}=0$,  then $E(w_\infty)=0$. Combining Lemma \ref{Monotonicity formula} (monotonicity formula), Proposition \ref{prop compactness} and the scaling invariance of the monotonicity quantity, we obtain
\[ E(w_\infty)=\lim_{k\to+\infty}E(1;0,0,u_k)=\lim_{s\to+\infty} E(s;0,0,u).\]
In view of Corollary \ref{coro 3.7}, this implies that for any $s>0$,
\[E(s;0,0,u)\equiv 0.\]
Applying Lemma \ref{Monotonicity formula} once again, we deduce that $u$  is backward self-similar. Let $w=u(-1)$ be the restriction of $u$ to the time $-1$ slice, which is a solution to the equation \eqref{SC1}. Testing this equaiton with $w$, we obtain
\[\int_{\R^n}\left[|\nabla w|^2+\frac{1}{p-1}w^2-|w|^{p+1}\right] \rho dy=0.\]
Combining this equality with $E(w)=0$, we deduce that $w\equiv 0$. Because $u$ is self-similar, we also have $u\equiv 0$.

If $w_{\infty}=\kappa,$ then we can apply Theorem \ref{maintheorem1} to get a constant $T\geq 0$ such that
\[u(x, t)=\kappa(T-t)^{-\frac{1}{p-1}},\quad\text{in}~\mathbb{R}^{n}\times (-\infty, 0). \]
The proof is complete.
\end{proof}

 \section{Proof of Theorem \ref{maintheorem2}}
 	\setcounter{equation}{0}

In this section, we   use Theorem \ref{maintheorem1} to prove Theorem \ref{maintheorem2}. The proof is basically a contradiction argument, to show that if the conclusions in Theorem \ref{maintheorem2} do not hold, then we can construct (by suitable rescaling) an ancient solutions violating Theorem \ref{maintheorem1}.

In the setting of Theorem \ref{maintheorem2}, $u$ is only defined in $Q_1^-$, so let us first modify it to a function defined in $\R^n\times(-1,0)$.  Let $\eta$ be a standard cut off function satisfying
$\eta=1$ in $B_{1/2}(0)$, $\eta=0$ outside
$B_1(0)$ and $|\nabla\eta|+|\Delta\eta|\leq C$.

Define \[\widetilde{u}=\eta u+(1-\eta)u_0,\]
where  $u_0=\kappa(-t)^{-\frac{1}{p-1}}$ is the ODE solution defined in \eqref{ODE solution}. A direct calculation shows that $\widetilde{u}$ satisfies
\begin{equation}\label{eqn modified}
\partial_{t}\widetilde{u}-\Delta\widetilde{u}=|\widetilde{u}|^{p-1}\widetilde{u}+h
\end{equation}
with
\begin{equation}\label{definition of h}
    h:= \eta |u|^{p-1}u+(1-\eta)u_{0}^{p}-|\widetilde{u}|^{p-1}\widetilde{u}-\Delta\eta u+\Delta\eta u_0-2\nabla\eta\cdot\nabla u.
\end{equation}
By the Morrey space estimate on $u$ and the explicit form of $u_0$, we get
\begin{equation}\label{estimate on h}
   \int_{Q_r^-(x,-r^2)}|h|^{\frac{p+1}{p}}\leq C r^{n+2-m}, \quad \forall x\in \R^n, ~  0<r<1.
\end{equation}
Hence if we define the monotonicity quantity $E(s;x_0,0,\widetilde{u})$ as in Section 3, and if $x_0\in B_{1/4}$, then there exists a universal constant $C$ such that
\begin{equation}\label{approximate monotonicity formula}
   E(s;x_0,0,\widetilde{u})+Ce^{-C/s}
\end{equation}
is an increasing function of $s$. With this modification, all results in Section 3 still hold for $\widetilde{u}$.

\begin{proof}[Proof of Theorem \ref{maintheorem2}]

\ref{item:uniqueness} The uniqueness of the tangent function has been established in Remark \ref{pro uniquenessofblowup}.

\ref{item:stability}  Assume by the contrary that there exists a sequence $\{(x_{k}, 0)\}\subset \text{Sing}(u)$ such that $\lim_{k\rightarrow\infty}x_{k}=0$, but for each $k$,
\begin{equation}\label{blowupsetscondition}
 \kappa\not\in\mathcal{T}(x_{k},0; u).
 \end{equation}

 Recall that $\varepsilon_1$ is the  constant in Theorem \ref{thm rigidity 1}. Because $\mathcal{T}(0,0;u)=\{\kappa\}$, by Proposition \ref{prop convergence for self-similar equations}, there exists $\lambda_\ast>0$ such that for all $\lambda\in(0,\lambda_\ast)$,  
 \[ \|u_{0,0,\lambda}(\cdot, -1)-\kappa\|_{L^2_w(\R^n)}<\varepsilon_1.\]
For any  $\lambda$ fixed, because $u_{x_k,0,\lambda}$ is just a translation of $u_{0,0,\lambda}$ in the spatial direction and $|x_k|/\lambda\to 0$ as $k\to\infty$, we see  
for all  $k$ large,
 \[ \|u_{x_k,0,\lambda}(\cdot, -1)-\kappa\|_{L^2_w(\R^n)}<\varepsilon_1.\]
On the other hand, combining \eqref{blowupsetscondition} with Theorem \ref{thm rigidity 1}, we deduce that for any $k$, there exists $\lambda_k^\prime$ such that
\[ \|u_{x_k,0,\lambda_k^\prime}(\cdot, -1)-\kappa\|_{L^2_w(\R^n)}>\varepsilon_1.\]
Then by the continuity of $u_{x_k,0,\lambda}(\cdot, -1)$ in $\lambda$ with respect to the $L^2_{w}(\R^n)$ distance (see Remark \ref{rmk continuity of blow-up sequence in L2}), there exists 
 a sequence $\{\lambda_{k}\}$ tending to  $0$ such that 
\[\|u_{x_k,0,\lambda}(\cdot, -1)-\kappa\|_{L^2_w(\R^n)}<\varepsilon_1, \quad\forall\lambda\in (\lambda_{k},1)\] 
and
 \[\|u_{x_k,0,\lambda_{k}}(\cdot, -1)-\kappa\|_{L^2_w(\R^n)}=\varepsilon_1.\] 
 If we set
 \begin{equation}\label{defineresacleuk}
u_{k}(x, t)=\lambda_{k}^{\frac{2}{p-1}}u(x_{k}+\lambda_{k}x, \lambda_{k}^{2}t),
\end{equation}
then 
	\begin{equation}\label{uktcondition}
		\left\{\begin{aligned}
			&\|(u_{k})_{0,0, \frac{\lambda}{\lambda_{k}}}(\cdot,-1)-\kappa\|_{L^2_w(\R^n)}< \varepsilon_1,   &\forall ~ \lambda\in (\lambda_{k},1), \\
			&\|(u_{k})_{0,0, 1}(\cdot,-1)-\kappa\|_{L^2_w(\R^n)}=\varepsilon_1.
		\end{aligned}\right.
	\end{equation}

For each $k$, the function $u_{k}$ defined in \eqref{defineresacleuk} is  a suitable weak solution of \eqref{eqn} in $B_{\lambda_k^{-1}/2}(0)\times \left(-\lambda_{k}^{-2}, 0\right)$.
Similar to Proposition \ref{prop compactness}, with the help of the Morrey estimate \eqref{Morreyeatimates} (which is scaling invariant), we get a function $u_{\infty}\in\mathcal{F}_M$ such that as $k\to\infty$,
 $u_{k}\to u_{\infty}$ in $\R^n\times(-\infty,0)$.
 
Passing to the limit in \eqref{uktcondition} leads to
\begin{equation}\label{uinftytcondition}
		\left\{\begin{aligned}
			&\|(u_{\infty})_{0, 0, -t}(\cdot,-1)-\kappa\|_{L^2_w(\R^n)}\leq \varepsilon_1,   &\forall ~t\in \left(-\infty, -1\right),\\
			&\|(u_{\infty})_{0, 0, 1}(\cdot, -1)-\kappa\|_{L^2_w(\R^n)}=\varepsilon_1.
		\end{aligned}\right.
	\end{equation}

Let $\{\mu_{k}\}$ be any sequence tending to $\infty$ as $k\to\infty$. For each $k$, we define 
\[u_{\infty, k}(x, t)=\mu_{k}^{\frac{2}{p-1}}u_{\infty}(\mu_{k}x, \mu_{k}^{2}t).\]
By Proposition \ref{prop blow-down limit}, there exists a function $w_{\infty}\in\mathcal{S}_{M}$ such that
\[u_{\infty, k}\to(-t)^{-\frac{1}{p-1}}w_{\infty}\left(\frac{x}{\sqrt{-t}}\right),\quad\text{as}~k\to\infty.\]
The first condition in \eqref{uinftytcondition} implies 
\[\|w_{\infty}-\kappa\|_{L^2_w(\R^n)}\leq \varepsilon_1.\]
By Theorem \ref{thm rigidity 1}, we must have $w_\infty\equiv\kappa$, in other words, the blow-down limit of $u_\infty$ is $\kappa$. Then by Theorem \ref{maintheorem1}, we deduce that  $u_{\infty}=\kappa(T-t)^{-\frac{1}{p-1}}$ for some $T\geq0$. 

Because $(x_k,0)$ is a singular point of $u$, $(0,0)$ is a singular point of $u_k$. Then by Corollary \ref{coro convergence of singular points}, $(0,0)$ is also a singular point of $u_\infty$. Thus $T=0$ and  $u_{\infty}=\kappa(-t)^{-\frac{1}{p-1}}$. As a consequence, $u_\infty(x,-1)\equiv\kappa$.
This gives a contradiction with the second condition in \eqref{uinftytcondition}. The proof  of \ref{item:stability} is thus complete.

\ref{item:positivity} and \ref{item:Type I}: These two items can be proved by the same argument, so here we only show how to get the Type I blow-up rate.

Assume by the contrary, there exists a sequence $(x_k,t_k)\to(0,0)$ such that either it is a singular point of $u$ or it is a regular point but
 \begin{equation}\label{absurd assumption to Type I}
     \lim_{k\to\infty}\frac{|u(x_k,t_k)|}{(-t_k)^{-\frac{1}{p-1}}}=+\infty.
 \end{equation}
Define
\[ \widetilde{u}_k(x,t):=(-t_k)^{\frac{1}{p-1}} u\left(x_k+(-t_k)^{\frac{1}{2}}x, (-t_k)t\right).\]
As in the proof of \ref{item:stability}, $\{\widetilde{u}_k\}$ converges to some $\widetilde{u}_\infty\in\mathcal{F}_M$ in $\mathbb{R}^n\times(-\infty,0)$. Combining  \eqref{absurd assumption to Type I} with Corollary \ref{cor L2 convergence lift to smooth},  we deduce that $(0,-1)$ is a singular point of $\widetilde{u}_\infty$. As a consequence, $\widetilde{u}_\infty$ cannot be ODE solutions.

Then arguing as in the proof of \ref{item:stability}, we find a sequence $\lambda_k\to0$ such that
\[ u_k(x,t):=\lambda_k^{\frac{2}{p-1}}u(x_k+\lambda_k x, \lambda_k^2 t)\]
satisfies \eqref{uktcondition}. This leads to a contradiction exactly in the same way as there.
\end{proof}

\section{Proof of Theorem \ref{maintheorem3}}
\setcounter{equation}{0}

This section is devoted to the proof of Theorem \ref{maintheorem3} and its corollaries.

\begin{proof}[Proof of Theorem \ref{maintheorem3}]
We divide the proof into two steps. 

In the first step we use Almgren stratification (as used by White \cite{White1997} for some other parabolic problems) to decompose the singular set $\text{Sing}(u)$ into two parts, $\mathcal{R}$ and $\mathcal{S}$. For any point in $\mathcal{R}$, any tangent function at it has enough symmetry, so the tangent function must be the ODE solution by the Liouville theorem of Giga-Kohn \cite{Giga-Kohn1985}. Then the openness of $\mathcal{R}$ is a consequence of the stability result in Theorem \ref{maintheorem2}. The Hausdorff dimension estimate on the closed set $\mathcal{S}$  follows from an application of Federer dimension reduction principle.

In the second step, we show that $\mathcal{R}$ is $(n-1)$-rectifiable, by using a theorem of Vel\'{a}zquez \cite{Velazquez1993}. Although this theorem is only stated for the subcritical case in \cite{Velazquez1993},  we observe that in Theorem \ref{maintheorem2}, we have already established a local Type I blow-up bound and the uniqueness of tangent functions. With these results in hands, the method in Vel\'{a}zquez \cite{Velazquez1993}  is still applicable to the current supercritical setting.

{\bf Step 1. Decomposition into $\mathcal{R}$ and $\mathcal{S}$.}  Let $u$ be a solution of the equation \eqref{eqn} satisfying the assumptions in Theorem \ref{maintheorem2}. Define the stratification of   the blow-up set  $\text{Sing}(u)$ to be
\[\mathcal{S}_{0}\subset\mathcal{S}_{1}\cdots\subset\mathcal{S}_{n}=\text{Sing}(u),\]
where $\mathcal{S}_{k}$ consists of  all blow-up points whose  tangent functions are at most translation invariant in $k$ directions. By \cite[Theorem 8.1]{White1997}, the Hausdorff dimension of $\mathcal{S}_k$ is at most $k$.

Set
\[\mathcal{R}=\text{Sing}(u)\backslash \mathcal{S}_{n-[m]-1}, \quad  \mathcal{S}=\mathcal{S}_{n-[m]-1}.\]
Then $\text{Sing}(u)=\mathcal{R}\cup \mathcal{S}$, and the Hausdorff dimension of  $\mathcal{S}$ is at most $n-[m]-1$.
Hence Item \ref{item:Hausdorff} is proved.

If $x\in\mathcal{R}$, then there is at least one tangent function $w$ which is  translation invariant in $n-[m]$ directions. Hence $w$ can be viewed as a solution of \eqref{SC1} in $\mathbb{R}^{[m]}$. Because in dimension $[m]$, $p$ is subcritical, by the Liouville theorem of Giga-Kohn \cite{Giga-Kohn1985}\footnote{This Liouville theorem is stated for bounded solutions of \eqref{SC1}, but because its proof involves only an application of Pohozaev identity, it also holds for solutions in $\mathcal{S}_M$. Indeed, by using the stationary condition in the definition of suitable weak solutions and the Morrey space bound \eqref{Morreyeatimates}, we can get the same Pohozaev identity.},  $w=\pm\kappa$. By Theorem \ref{maintheorem2},   $\mathcal{T}(x, 0; u)=\{\kappa\}$ or $\{-\kappa\}$. This allows us to define $\mathcal{R}^\pm$. Furthermore, Theorem \ref{maintheorem2} \ref{item:stability} implies that there is a neighborhood of $x$ satisfying this property. This proves the openness of $\mathcal{R}^\pm$, that is, Item \ref{item:openness}.

{\bf Step 2. Rectifiablity of $\mathcal{R}$.} 
It remains to prove Item \ref{item:rectifibility}. Since the case $x\in\mathcal{R}^{-}$ can be discussed similarly, we will focus on the case $x\in\mathcal{R}^{+}$.

 We get from Theorem \ref{maintheorem2} that  for any $x\in\mathcal{R}^{+}$, there exist two positive constants $\delta$ and $C$ such that, for any $(y,t)\in Q_\delta^-(x,0)$,
    \[0<u(y,t)\leq C(-t)^{-\frac{1}{p-1}}.\] Moreover, if $y\in B_{\delta}(x)\cap\mathcal{R}^{+}$, then $\mathcal{T}(y,0; u)=\{\kappa\}$. After a translation and scaling, we may assume that $x=0$ and $\delta=1$.

Let $\widetilde{u}$ be the modification of $u$ introduced at the beginning of  the previous section, with $h$ as defined in \eqref{definition of h}. 
By Theorem \ref{maintheorem2} \ref{item:Type I} and standard parabolic regularity theory, there exists a positive constant $C$ such that
\[|h|\leq C(-t)^{-\frac{p}{p-1}},\quad\text{in}~\mathbb{R}^{n}\times (-1, 0).\]
Moreover, $\text{spt}(h)\subset (B_{1}(0)\backslash B_{1/2}(0))\times (-\infty, 0)$.

Take an arbitrary point $(x,0)\in\text{Sing}(u)\cap B_{1/4}(0)$.
Let
$\widetilde{w}$ be the self-similar transform of $\widetilde{u}$ with respect to $(x,0)$.
Then $\widetilde{w}$ satisfies
\[\partial_{\tau}\widetilde{w}-\Delta\widetilde{w}+\frac{y}{2}\cdot\nabla\widetilde{w}+\frac{1}{p-1}\widetilde{w}=\widetilde{w}^{p}+g,\]
where $g$ is a bounded function satisfying $\text{spt}(g(\cdot, \tau))\subset\{e^{\frac{\tau}{2}}\leq |y|\leq 2e^{\frac{\tau}{2}}\}$.

Let $\psi=\widetilde{w}-\kappa$. It satisfies
\begin{equation}\label{linearizedequation}
\partial_{\tau}\psi-\Delta\psi+\frac{y}{2}\cdot\nabla\psi-\psi=f(\psi)+g,
\end{equation}
where
\[f(\psi)=(\psi+\kappa)^{p}-\kappa^{p}-p\kappa^{p-1}\psi.\]
Because the tangent function of $u$ at $(x,0)$ is $\kappa$,  $\psi(\cdot, \tau)\to 0$
as $\tau\to+\infty$ in $C_{loc}^{\infty}(\mathbb{R}^{n})$.

Recall that the operator $A:=-\Delta+\frac{y}{2}\cdot\nabla$ is a self-adjoint operator on $L_\omega^2(\mathbb R^n)$ with domain $D(A):=H_{\omega}^1(\mathbb R^n)$. Its spectrum consists of  eigenvalues
\begin{equation*}\label{110}
\lambda_\alpha= \frac{|\alpha|}{2},\quad |\alpha|=0,~1,~2...
\end{equation*}
and the corresponding eigenfunctions are given by
\begin{equation*}\label{111}
H_\alpha(y)=H_{\alpha_1}(y_1)...H_{\alpha_n}(y_n),
\end{equation*}
where $H_{\alpha_i}(y_i)=\bar{c}_{\alpha_i}\widetilde{H}_{\alpha_i}(y_i/2)$, $\bar{c}_{\alpha_i}=\Big(2^{\alpha_i/2}(4\pi)^{1/4}(\alpha_i!)\Big)^{-1}$ and $\widetilde{H}_{\alpha_i}(y_i)$ is the standard Hermite polynomial. It is well known that $\{H_\alpha\}$ is an orthonormal basis of $L_\omega^2(\mathbb R^n)$. 

The following theorem is the main result of Vel\'{a}zquez \cite{Velazquez1993}. 
\begin{thm}\label{thm next order expansion near ODE blow-up}
If $\psi\not\equiv 0$, we have the following two possibilities. 
\begin{enumerate}
    \item 
Either there exists an orthogonal transformation of coordinate axes,  still denoted by $y$, such that
\begin{equation}\label{112}
\psi(y,\tau)=\frac{C_p}{\tau}\sum_{k=1}^{\ell}H_2(y_k)+o\left(\frac{1}{\tau}\right),~\text{as}~\tau\to\infty,
\end{equation}
where $1\leq \ell\leq n$ and $C_p=((4\pi)^{1/4}(p-1)^{-1/(p-1)})/\sqrt{2}p$, 
\item or there exist an even number $m\geq 4$ and constants $C_\alpha$ not all zero such that
\begin{equation}\label{113}
\psi(y,\tau)=e^{(1-\frac{m}{2})\tau}\sum_{|\alpha|=m}C_\alpha H_\alpha(y)+o\left(e^{(1-\frac{m}{2})\tau}\right),~\text{as}~\tau\to\infty,
\end{equation}
where the homogeneous multilinear form $B(x)=\sum_{|\alpha|=m}C_\alpha  x^\alpha$ is nonnegative.
\end{enumerate}
Moreover, the convergence in both cases takes place in $H_{\omega}^1(\mathbb R^n)$ and $C_{loc}^\infty(\mathbb R^n)$.
\end{thm}
With the Type I blow-up rate \eqref{Type I blow-up rate}, the proof in \cite{Velazquez1993} (which is stated only for the subcritical case) can be extended to our setting without any change. Therefore, the remaining proof on the $(n-1)$-rectifiability of $\mathcal{R}^\pm$ is identical to the one of \cite[Proposition 3.1]{Velazquez1993-2}.
\end{proof}

Finally, we give the proof of the corollaries in Subsection \ref{subsection application}.

\begin{proof}[Proof of Corollary \ref{coro second Liouville}]
We prove this corollary by induction on the dimension, which is essentially Federer's dimenion reduction principle.  By Proposition \ref{prop blow-down limit}, take a $w_1\in\mathcal{S}_M$ (with the associated $u_1$) to be a blow-down limit of $u$.

If $p$ is subcritical in dimension $n$, then by Giga-Kohn's Liouville theorem, $u_1$ is an ODE solution. By Merle-Zaag's theorem \cite{Merle-Zaag}, $u$ itself is also an ODE solution. Because $u$ blows up at $t=0$, we must have
\[ u(x,t)\equiv\pm\kappa(-t)^{-\frac{1}{p-1}}.\]

Next, assume that the conclusion holds in dimension $n-1$. We want to show that it also holds in dimension $n$.
By the assumption that $u$ blows up on the whole $\R^n\times\{0\}$, and because singular points converges to singular points when we take the blow-down limits, we deduce that $u_1$ also blows up on the whole $\R^n\times\{0\}$. Because $u_1$ is backward self-similar,  for any $x_1\neq 0$, if we blow-up $u_1$ at $(x_1,0)$,  we get a backward self-similar solution $u_2$ which is translational invariant along the $x_1$ direction. In other words, $u_2$ can be viewed as a solution in dimension $n-1$. Then by the inductive hypothesis, $u_2(x,t)\equiv \pm\kappa(-t)^{-\frac{1}{p-1}}$. By Theorem \ref{maintheorem2}, any nonzero $x$ belongs to $\mathcal{R}^\pm$. But Theorem \ref{maintheorem3} says, unless $u_1(x,t)\equiv \pm\kappa(-t)^{-\frac{1}{p-1}}$, its blow-up set would be of dimension at most $n-1$. Hence it is only possible that $u_1$ is an ODE solution. Then
an application of Theorem \ref{maintheorem1} implies that $u$ is also an ODE solution.
\end{proof}

\begin{proof}[Proof of Corollary \ref{coro nearly subcritical case}]
Let $u(x, t)=(-t)^{-\frac{1}{p-1}}w(\frac{x}{\sqrt{-t}})$, which
 is an ancient solution of \eqref{eqn} in $\mathcal{F}_M$. Take a point   $x$ on the unit sphere $\mathbb{S}^{n-1}$.

If $u$ blows up at the point $(x, 0)$. By the Federer dimension reduction principle, we can blow-up it at $(x, 0)$
to get a function  $\widetilde{w}\in\mathcal{S}_{M}$. Then $\widetilde{w}$ is translation invariant  in one direction. In particular, $\widetilde{w}$ is a solution of \eqref{SC1} in $\mathbb{R}^{n-1}$.
Because $1<p<\frac{n+1}{n-3}$, an application of Giga-Kohn's Liouville theorem shows that $\widetilde{w}\equiv\pm\kappa$. By Theorem \ref{maintheorem2}, there exist two positive constant $C, \delta_{x}$ such that
\begin{equation}\label{11.11}
|u|\leq C(-t)^{-\frac{1}{p-1}},\quad\text{in}~B_{\delta_{x}}(x)\times (-\delta^{2}_{x}, 0).
\end{equation}

 If $u$ is regular at $(x, 0)$,  there also exists a constant $\delta_{x}$ such that $u$ is bounded
in $B_{\delta_{x}}(x)\times (-\delta^{2}_{x}, 0)$.  Thus \eqref{11.11} trivially holds for  $(x,0)$.

Since $\mathbb{S}^{n-1}$ is compact, by a covering, we  find two positive constants $C, \bar{\delta}$ such that
\[|u|\leq C(-t)^{-\frac{1}{p-1}},\quad\text{in}~(B_{1+\bar{\delta}}(0)\backslash B_{1-\bar{\delta}}(0))\times (-\bar{\delta}, 0).\]
Coming back to $w$, this says $w$ is bounded at infinity. Since $w$ is smooth, it is bounded in the entire $\mathbb{R}^{n}$.
\end{proof}

\begin{proof}[Proof of Corollary \ref{coro radially symmetric cae}]
Let $u(x, t)=(-t)^{-\frac{1}{p-1}}w(\frac{x}{\sqrt{-t}})$, which is
 an ancient solution of \eqref{eqn} in $\mathcal{F}_M$. Assume $w$ is not the constant solution and $u$ blows up at some point $x\in\mathbb{R}^{n}\backslash\{0\}$. Because $u(x,0)$ is $(-\frac{2}{p-1})$-homogeneous and radially symmetric, the singular set of $u$ contains $\mathbb{R}^{n}\backslash\{0\}$. Moreover, for each $x\neq 0,\mathcal{T}(x, 0, u)=\{\kappa\}$ or $\{-\kappa\}$. This is a contradiction with Theorem \ref{maintheorem3}. Therefore, there exist two constant $C, \bar{\delta}$ such that $u$ is bounded in $(B_{1+\bar{\delta}}(0)\backslash B_{1-\bar{\delta}}(0))\times (-\bar{\delta}, 0)$. Coming back to $w$, this is exactly \eqref{decay estimate for radially symmetric soluitons}.
\end{proof}

\appendix
\section{Proof of Proposition  \ref{pro2}}\label{sec proof of Proposition 4.4}
	\setcounter{equation}{0}

The proof of Proposition \ref{pro2} needs several lemmas.
\begin{lem}\label{lempre}
For any $R_{0}>1$ and $\varepsilon>0$, there exists $R_{1}\geq R_{0}$  such that if $R>R_{1}$, $w\in\mathcal{S}_M$ is smooth in $\overline{B_R(0)}$ and it satisfies
\[\text{dist}_{R}(w,\kappa)+\text{dist}_{R}\left(\Lambda(w), \frac{2}{p-1}\kappa\right)\leq \frac{\kappa}{p+1},\]
then
\begin{equation}\label{lempre1}
\text{dist}_{R_{0}}(w, \kappa)+\text{dist}_{R_{0}}\left(\Lambda(w), \frac{2}{p-1}\kappa\right)\leq \varepsilon.
\end{equation}
\end{lem}
\begin{proof}
Assume by the contrary that there exist $R_0$ and $\varepsilon$,  and two sequences $w_i\in\mathcal{S}_M$, $R_i\to+\infty$ such that $w_i$ is smooth in $\overline{B_{R_i}(0)}$,
\begin{equation}\label{lemprea}
\begin{aligned}
\text{dist}_{R_i}(w_i,\kappa)+\text{dist}_{R_i}\left(\Lambda(w_i), \frac{2}{p-1}\kappa\right)\leq \frac{\kappa}{p+1},
\end{aligned}
\end{equation}
but
\begin{equation}\label{lempre2}
\text{dist}_{R_{0}}(w_i, \kappa)+\text{dist}_{R_{0}}\left(\Lambda(w_i), \frac{2}{p-1}\kappa\right)>\varepsilon.
\end{equation}

By \eqref{lemprea} and   standard elliptic regularity theory,  there exists a function $w_{\infty}$ such that $w_i\to w_{\infty}$  in $C^{\infty}_{loc}(\mathbb{R}^{n})$. In particular, $w_{\infty}$ is a  smooth solution of \eqref{SC1}. Taking limit  in \eqref{lemprea} leads to
\begin{equation}\label{functiondistance}
|w_{\infty}-\kappa|+
\left|\Lambda(w_{\infty})-\frac{2}{p-1}\kappa\right|\leq \frac{\kappa}{p+1}\quad\text{in}~\mathbb{R}^{n}.
\end{equation}
Then by Proposition \ref{ChCS},   $w_{\infty}\equiv \kappa$.

On the other hand,  by the $C^\infty_{loc}(\R^n)$ convergence of $w_i$, taking limit   in \eqref{lempre2} gives
\begin{equation}\label{distancegap}
\text{dist}_{R_{0}}(w_{\infty}, \kappa)+\text{dist}_{R_{0}}\left(\Lambda(w_{\infty}), \frac{2}{p-1}\kappa\right)\geq \varepsilon.
\end{equation}
This is a contradiction.
\end{proof}
The above lemma implies that on a fixed compact subset, we can make $w$ as close to $\kappa$ as possible. In order to finish the proof of Proposition \ref{pro2}, it suffices to show that both $\nabla w$ and $\nabla \Lambda(w)$ are exponentially small in $R$. For this purpose, we define
\[\tau:=\frac{w^p}{\Lambda(w)}.\] 
This function is well defined because both $w$ and $\Lambda(w)$ are continuous, positive functions on $B_{R}(0)$.

The auxiliary function $\tau$ satisfies the following equations.
\begin{lem}\label{formulas}
The function $\tau$ satisfies
\begin{align}\label{formula1}
\frac{1}{\Lambda(w)^{2}\rho}{\rm\text{div}}\left(\Lambda(w)^{2}\rho\nabla \tau\right)=\frac{p(p-1)w^{p-2}|\nabla w|^{2}}{\Lambda(w)}
\end{align}
and
\begin{align}\label{formula2}
\frac{1}{\Lambda(w)^{2}\rho}{\rm\text{div}}\left(\Lambda(w)^{2}\rho\nabla \tau^{2}\right)=2|\nabla \tau|^{2}+\frac{2 p(p-1)\tau w^{p-2}|\nabla w|^{2}}{\Lambda(w)}.
\end{align}
\end{lem}
\begin{proof}
Since $w$ is a solution of \eqref{SC1}, $w^p$ satisfies the equation
\begin{align*}\label{pro2.1}
\Delta w^p-\frac{y}{2}\cdot\nabla w^p+\frac{1}{p-1}w^p+pw^{p-1}w^p-w^p=p(p-1)w^{p-2}|\nabla w|^{2}.
\end{align*}
Moreover, we get from \eqref{SC1} that $\Lambda(w)$ satisfies
\begin{equation}\label{pro2.2}
\Delta\Lambda(w)-\frac{y}{2}\cdot\nabla \Lambda(w)-\frac{1}{p-1}\Lambda(w)+pw^{p-1}\Lambda(w)-\Lambda(w)=0.
\end{equation}

By a direct computation, we obtain
\begin{equation*}\label{pro2.3}
\begin{aligned}
&\frac{1}{\Lambda(w)^{2}\rho}{\rm\text{div}}\left(\Lambda(w)^{2}\rho\nabla \tau\right)\\
=&\Delta\tau+\frac{2}{\Lambda(w)}\nabla \Lambda(w)\cdot\nabla\tau-\frac{y}{2}\cdot\nabla\tau\\
=&\frac{\Lambda(w)\Delta w^p-w^p\Delta \Lambda(w)}{\Lambda(w)^{2}}-\frac{y}{2}\cdot\nabla\tau\\
=& \Lambda(w)^{-1}\left[\frac{y}{2}\cdot \nabla w^p+\frac{p}{p-1}w^p-pw^{p-1}w^p+p(p-1)w^{p-2}|\nabla w|^{2}\right]\\
&-\frac{w^p}{\Lambda(w)^{2}}\left[\frac{y}{2}\cdot\nabla\Lambda(w)+\frac{1}{p-1}\Lambda(w)-pw^{p-1}\Lambda(w)+\Lambda(w)\right]-\frac{y}{2}\cdot\nabla\tau\\
=&\frac{p(p-1)w^{p-2}|\nabla w|^{2}}{\Lambda(w)},
\end{aligned}
\end{equation*}
which is \eqref{formula1}.

Since
\[\frac{1}{\Lambda(w)^{2}\rho}{\rm\text{div}}\left(\Lambda(w)^{2}\rho\nabla \tau^{2}\right)=\frac{2\tau}{\Lambda(w)^{2}\rho}{\rm\text{div}}(\Lambda(w)^{2}\rho\nabla \tau)+2|\nabla\tau|^{2},\]
  \eqref{formula2} follows from \eqref{formula1}.
\end{proof}
By   Lemma \ref{formulas}, we can obtain  the following integral estimate on $|\nabla w|$ and $|\nabla\Lambda(w)|$.
\begin{lem}\label{lem3.1}
For any $s\in (0, R)$, we have
\begin{equation}\label{lem3.1.1}
\int_{B_{R-s}(0)}\left[|\nabla \tau|^{2}+\frac{\tau w^{p-2}|\nabla w|^{2}}{\Lambda(w)}\right]\Lambda(w)^{2}\rho dy\leq\frac{C}{s^{2}}R^{n}e^{-\frac{(R-s)^{2}}{4}}\sup_{y\in B_{R}(0)}w(y)^{2p},
\end{equation}
where $C$ is a positive constant independent of $R$.
\end{lem}
\begin{proof}
For any $\phi\in C_0^\infty(B_{R}(0))$, using the divergence theorem, we get
\begin{equation*}\label{lem3.1.2}
\begin{aligned}
0=&\int_{B_{R}(0)}{\rm \text{div}}\left(\Lambda(w)^{2}\rho\phi^{2}\nabla\tau^{2}\right)dy\\
=&2\int_{B_{R}(0)}\phi^{2}\left[|\nabla \tau|^{2}+\frac{p(p-1)\tau w^{p-2}|\nabla w|^{2}}{\Lambda(w)}\right]\Lambda(w)^{2}\rho dy\\
+&4\int_{B_{R}(0)}\phi\tau\nabla\phi\cdot\nabla\tau\Lambda(w)^{2}\rho dy.
\end{aligned}
\end{equation*}
Applying  Cauchy-Schwarz inequality $4ab\leq a^{2}+4b^{2}$, we obtain
\begin{equation*}\label{lem3.1.2'}
 4\int_{B_{R}(0)}\phi\tau\nabla\phi\cdot\nabla\tau\Lambda(w)^{2}\rho dy\leq \int_{B_{R}(0)}|\nabla \tau|^{2}\phi^{2}\Lambda(w)^{2}\rho dy+4\int_{B_{R}(0)}\tau^{2}|\nabla\phi|^{2}\Lambda(w)^{2}\rho dy.
\end{equation*}
Hence there exists a universal positive constant $C$ such that
\begin{equation}\label{lem3.1.3}
\int_{B_{R}(0)}\phi^{2}\left[|\nabla \tau|^{2}+\frac{\tau w^{p-2}|\nabla w|^{2}}{\Lambda(w)}\right]\Lambda(w)^{2}\rho dy\leq C\int_{B_{R}(0)}\tau^{2}|\nabla\phi|^{2}\Lambda(w)^{2}\rho dy.
\end{equation}
Choose  $\phi$ satisfying
$\phi\equiv 1$ in $B_{R-s}(0)$ and 
$|\nabla\phi|\leq 2/s$ in $B_{R}(0)\backslash B_{R-s}(0)$.
 Then
 \eqref{lem3.1.1} follows  from \eqref{lem3.1.3}, if we note that by definition, $\tau\Lambda(w)=w^p$.
\end{proof}
Combing the integral estimates of Lemma \ref{lem3.1}  with  standard elliptic estimates, 
we can establish the following pointwise  bound on $|\nabla w|$ and $|\nabla\Lambda(w)|$.
\begin{lem}\label{lem 4.1}
There exists a universal constant $C$   such that
\begin{equation}\label{cor4.1.1}
\sup_{y\in B_{R-3}(0)}(|\nabla\Lambda(w)|^{2}+|\nabla w|^{2})\leq CR^{2n}e^{-\frac{R}{4}}.
\end{equation}
\end{lem}
\begin{proof}
Lemma \ref{lem3.1} with $s=1/2$ gives
\begin{equation}\label{cor4.1.2}
\int_{B_{R-\frac{1}{2}}(0)}\left[|\nabla \tau|^{2}+|\nabla w|^{2}\right]\rho dy\leq CR^{n}\exp\left\{-\frac{(R-\frac{1}{2})^{2}}{4}\right\}.
\end{equation}
Plugging the relation
\[\nabla\tau=\frac{pw^{p-1}}{\Lambda(w)}\nabla w -\frac{w^p}{\Lambda(w)^{2}}\nabla\Lambda(w)\]
into \eqref{cor4.1.2}, we get
\begin{equation*}\label{cor4.1.2'}
\int_{B_{R-\frac{1}{2}}(0)}\left[|\nabla \Lambda(w)|^{2}+|\nabla w|^{2}\right]\rho dy\leq CR^{n}\exp\left\{-\frac{(R-\frac{1}{2})^{2}}{4}\right\}.
\end{equation*}
Since
\[e^{-\frac{|x|^{2}}{4}}\geq \exp\left\{-\frac{R^{2}-2R+1}{4}\right\},\quad\text{on}~B_{R-1}(0),\]
it follows that
\begin{equation}\label{cor4.1.3}
\int_{B_{R-1}(0)}\left[|\nabla \Lambda(w)|^{2}+|\nabla w|^{2}\right]dy\leq CR^{n}\exp\left\{-\frac{R}{4}\right\}.
\end{equation}
This is the integral decay on $\nabla\Lambda(w)$ and $\nabla w$.

Then we   combine \eqref{cor4.1.3} with elliptic regularity theory to get the pointwise bound. Let $\mathcal{L}$ be the linear operator defined by
\[\mathcal{L}\psi=\Delta\psi-\frac{y}{2}\cdot\nabla\psi-\frac{1}{p-1}\psi+pw^{p-1}\psi.\]
Since $w$ satisfies \eqref{SC1}, for each $i=1$, $\dots$, $n$,
\begin{equation*}\label{formula3}
\mathcal{L}\partial_{x_i}w-\frac{1}{2}\partial_{x_i}w=0,\quad\text{in}~B_{R}(0).
\end{equation*}
Moreover, we get from \eqref{pro2.2} that for $i=1, 2, \cdots n$,
\begin{equation*}\label{formula4}
\mathcal{L}\Lambda(w)_{i}-\Lambda(w)_{i}-\frac{1}{2}\Lambda(w)_{i}+p(p-1)w^{p-2}w_{i}\Lambda(w)=0,\quad\text{in}~B_{R}(0).
\end{equation*}
 Notice that the first order term in $\mathcal{L}$ grows at most linearly. By applying standard elliptic regularity theory  (see \cite[Theorem 9.20]{Gi-Tr}) on  $B_{1/R}(x)$  (for any $x\in B_{R-2}(0)$), we get
\begin{equation}\label{cor4.1.4}
|\nabla w(x)|^{2}\leq CR^{n}\int_{B_{1/R}(x)}|\nabla w|^{2}
\end{equation}
and
\begin{equation}\label{cor4.1.5}
|\nabla\Lambda(w)(x)|^{2}\leq CR^{n}\int_{B_{1/R}(y_{0})}|\nabla \Lambda(w)|^{2}+CR^{2n-2}e^{-\frac{R}{4}}.
\end{equation}
Combining \eqref{cor4.1.4}, \eqref{cor4.1.5} with the integral bound \eqref{cor4.1.3} gives \eqref{cor4.1.1}.
\end{proof}

Now  we are   ready to  prove  Proposition \ref{pro2}.
\begin{proof}[Proof of Proposition \ref{pro2}]
Take $R_{0}=5\sqrt{n}$ and $\varepsilon=\frac{\kappa}{100(p+1)}$ in Lemma \ref{lempre}. Then
\begin{equation}\label{prpro2.1}
\text{dist}_{5\sqrt{n}}(w, \kappa)+\text{dist}_{5\sqrt{n}}\left(\Lambda(w), \frac{2}{p-1}\kappa\right)\leq \frac{\kappa}{100(p+1)}.
\end{equation}
For any $x\in B_{R-3}(0)$ with $|y_{0}|>5\sqrt{n}$, by Lemma  \ref{lem 4.1},
\begin{equation}\label{prpro2.2}
\left|w(x)-w\left(\frac{y_{0}}{|y_{0}|}\right)\right|\leq \left| \int_{1}^{|y_{0}|}\frac{y_{0}}{|y_{0}|}\cdot\nabla w\left(\xi\frac{y_{0}}{|y_{0}|}\right)d\xi\right|\leq CR^{n+1}e^{-R/8}.
\end{equation}
 Combining \eqref{prpro2.1} and \eqref{prpro2.2}, we get
\[|w-\kappa|\leq \frac{\kappa}{100(p+1)}+ CR^{n+1}e^{-\frac{R}{8}}\leq \frac{\kappa}{20(p+1)} \quad\text{in}~B_{R-3}(0), \]
provided that $R$ is large enough. 

In the same way,  if $R$ is large enough, then
\[\text{dist}_{R-3}\left(\Lambda(w), \frac{2}{p-1}\kappa\right)\leq \frac{\kappa}{20(p+1)}. \qedhere\]
\end{proof}

\section{Proof of Proposition \ref{pro1}} \label{sec proof of proposition 4.5}
	\setcounter{equation}{0}

In this appendix, we prove Proposition \ref{pro1}. The main observation here is that by shrinking $B_{R}(0)$ slightly, we can choose a universal constant $s_{\ast}$ so that the $\varepsilon$-regularity result, Proposition \ref{epsilonregularity} is applicable. Now we give the details.

Throughout this appendix,   $w$ denotes a solution in $\mathcal{S}_M$ satisfying the assumptions in Proposition \ref{pro1}. By definition, $u(x, t)=(-t)^{-1/(p-1)}w(x/\sqrt{-t})\in\mathcal{F}_M$.

The following result, which is essentially \cite[Proposition 1']{Giga-Kohn1985}, gives  a gradient estimate  for solutions of \eqref{SC1}.  
\begin{lem}\label{lem gradient estimate for self-similar solutions}
There exists a universal  constant $C$  such that 
\[|\nabla w|+|\nabla^{2}w|+|\nabla^{3}w|\leq C,\quad\text{in}~B_{R-1}(0).\]
\end{lem}
\begin{proof}
For any $x\in B_{R-1}$, by the definition of $u$, we have
\[ |u(x+y,-1+t)|\leq C   \quad \text{for}~~ (y,t)\in Q_1^-.\]
The gradient estimate on $w$ then follows from standard parabolic estimates applied to $u$.
\end{proof}

\begin{lem}\label{lem small of monotonicity quantity}
 There exists  a universal, small constant $s_\ast$ such that for any $(x,t)\in B_{R-4}(0)\times[-1,-1+s_\ast]$, we have
\begin{equation}
\bar{E}(s_\ast; x, t, u)<\varepsilon_{0},
\end{equation}
where $\varepsilon_{0}$ is the positive constant in Proposition \ref{epsilonregularity}.
\end{lem}
\begin{proof}
We only estimate   the integral
\[\frac{1}{s_\ast}\int_{t-2s_\ast}^{t-s_\ast}(t-\tau)^{\frac{p+1}{p-1}}\int_{\R^n}|u(y,\tau)|^{p+1} G(x-y,t-\tau)dy d\tau\]
for $(x,t)\in B_{R-4}(0)\times[-1,-1+s_\ast]$,
where $s_\ast$ is a small constant to be determined below.
The other two terms in $\bar{E}(s_\ast;x,t,u)$ can be estimated in the same way.

Decompose this integral into two parts, $\mathrm{I}$ being on $B_1(x)$ and $\mathrm{II}$ on $\R^n\setminus B_1(x)$.

By Lemma \ref{lem gradient estimate for self-similar solutions}, $|u(y,\tau)|\leq C$ in $B_1(x)\times[t-2s_\ast,t-s_\ast]$. Thus
\begin{align*}
  \mathrm{I} &\leq Cs_\ast^{-1} \int_{t-2s_\ast}^{t-s_\ast} (t-\tau)^{\frac{p+1}{p-1}}\int_{\R^n} G(x-y,t-\tau)dy d\tau\\
  &\leq Cs_\ast^{-1} \int_{t-2s_\ast}^{t-s_\ast} (t-\tau)^{\frac{p+1}{p-1}} d\tau\\
  & \leq Cs_\ast^{\frac{p+1}{p-1}}\\
  &\leq \frac{\varepsilon_0}{6},
\end{align*}
provided $s_\ast$ is small enough.

Concerning $\mathrm{II}$, we take the decomposition
\begin{align}\label{estiamte on II, sec 4}
\mathrm{II}&= \sum_{i=1}^{+\infty} \frac{1}{s_\ast}\int_{t-2s_\ast}^{t-s_\ast} (t-\tau)^{\frac{p+1}{p-1}}\int_{B_{(i+1)}\setminus B_{i}}|u(y,\tau)|^{p+1} G(x-y,t-\tau)dy d\tau  \nonumber\\
&\leq Cs_\ast^{-\frac{n}{2}+\frac{2}{p-1}}  \sum_{i=1}^{+\infty} \exp\left(-\frac{i^2}{8s_\ast}\right)  \int_{t-2s_\ast}^{t-s_\ast}\int_{B_{(i+1)}(x)\setminus B_{i}(x)}|u(y,\tau)|^{p+1} dyd\tau.
\end{align}
We can cover $(B_{(i+1)}(x)\setminus B_{i}(x))\times(t-2s_\ast,t-s_\ast)$ by at most $Ci^n s_\ast^{-\frac{n}{2}}$ backward parabolic cylinders of size $\sqrt{s_\ast}$. Then by summing the Morrey estimate \eqref{Morreyeatimates} on these cylinders, we obtain
\[\int_{t-2s_\ast}^{t-s_\ast}\int_{B_{(i+1)}(x)\setminus B_{i}(x)}|u(y,\tau)|^{p+1} dyd\tau\leq Ci^n  s_\ast^{\frac{n}{2}+1-\frac{p+1}{p-1}}.\]
Plugging this estimate into \eqref{estiamte on II, sec 4}, we get
\begin{align*}
   \mathrm{II} &\leq C \sum_{i=1}^{+\infty} \exp\left(-\frac{i^2}{8s_\ast}\right) i^n \\
    & \leq Cs_\ast^{\frac{n}{2}} \int_{s_\ast^{-1/2}}^{+\infty}  \tau^n e^{-\frac{\tau^2}{8}}d\tau\\
    &\leq Cs_\ast^{\frac{n}{2}}e^{-\frac{1}{16s_\ast}}\\
    &\leq \frac{\varepsilon_0}{6},
\end{align*}
provided $s_\ast$ is small enough.

The proof is complete by putting the estimates on $\mathrm{I}$ and $\mathrm{II}$ together.
\end{proof}

\begin{proof}[Proof of Proposition \ref{pro1}]
The previous lemma allows us to use  Proposition \ref{epsilonregularity} in $Q_{\sqrt{s_\ast}}^-(x,t)$, for any $(x,t)\in B_{R-3}\times[-1,-1+s_\ast]$. This gives
\[|u| \leq C,\quad\text{in}~B_{R-4}(0)\times \left[-1-\frac{s_\ast}{2}, -1+s_\ast\right].\]

Then by standard parabolic regularity theory,  there exists a universal  constant $C$  such that
\[ |\partial_tu|+|\partial_{tt}u|\leq C,\quad\text{in}~B_{R-5}(0)\times \left[-1, -1+s_\ast\right].\]
An integration in $t$ gives, for any $x\in B_{R-5}(0)$,
\begin{equation}\label{B.1}
|u(x,-1)-u(x,-1+s_\ast)|+|\partial_tu(x,-1)-\partial_tu(x,-1+s_\ast)|\leq Cs_\ast.
\end{equation}

Recall that $u(x,t)=(-t)^{-\frac{1}{p-1}}w\left(\frac{x}{\sqrt{-t}}\right)$, so
\begin{align*}
\partial_tu(x,t)&=(-t)^{-\frac{p}{p-1}}\left[\frac{1}{p-1}w\left(\frac{x}{\sqrt{-t}}\right)+\frac{1}{2}\frac{x}{\sqrt{-t}}\cdot\nabla w\left(\frac{x}{\sqrt{-t}}\right)\right]\\
&=(-t)^{-\frac{p}{p-1}}\Lambda(w)\left(\frac{x}{\sqrt{-t}}\right).
\end{align*} 
In particular, 
\[u(x,-1)=w(x) \quad \text{and} \quad \partial_tu(x,-1)=\Lambda(w)(x).\]
Plugging these identities into \eqref{B.1}, we obtain, for any $x\in B_{R-5}(0)$,
\begin{align*}
    &\left|w(x)-(1-s_\ast)^{-\frac{1}{p-1}} w\left(\frac{x}{\sqrt{1-s_\ast}}\right)\right|\\
&\quad +\left|\Lambda(w)(x)-(1-s_\ast)^{-\frac{p}{p-1}} \Lambda(w)\left(\frac{x}{\sqrt{1-s_\ast}}\right)\right|\leq Cs_\ast.
\end{align*}
Combining this inequality with the assumption, which says
\[\text{dist}_{R}(w,\kappa)+\text{dist}_{R}\left(\Lambda(w), \frac{2}{p-1}\kappa\right)\leq \frac{\kappa}{10(p+1)},\]
we get
 \begin{align*}
    &\left|(1-s_\ast)^{-\frac{1}{p-1}} w\left(\frac{x}{\sqrt{1-s_\ast}}\right)-\kappa\right|\\
&\quad +\left| (1-s_\ast)^{-\frac{p}{p-1}} \Lambda(w)\left(\frac{x}{\sqrt{1-s_\ast}}\right)-\frac{2}{p-1}\kappa\right|\leq \frac{\kappa}{10(p+1)}+Cs_\ast.
\end{align*}
First, this implies that
\[ \sup_{x\in B_{R-5}(0)} \left| w\left(\frac{x}{\sqrt{1-s_\ast}}\right)\right| +\left|  \Lambda(w)\left(\frac{x}{\sqrt{1-s_\ast}}\right)\right|\leq \frac{10\kappa}{p-1}+Cs_\ast.\]
Then an Taylor expansion of $(1-s_\ast)^{-\frac{1}{p-1}}$ and $(1-s_\ast)^{-\frac{p}{p-1}}$ gives
 \begin{equation}\label{B2}
  \left| w\left(\frac{x}{\sqrt{1-s_\ast}}\right)-\kappa\right|\\
 +\left|   \Lambda(w)\left(\frac{x}{\sqrt{1-s_\ast}}\right)-\frac{2}{p-1}\kappa\right|\leq \frac{\kappa}{10(p+1)}+Cs_\ast.
\end{equation}

Choose $s_\ast$  small so that
\[\frac{\kappa}{10(p+1)}+Cs_\ast\leq \frac{\kappa}{p+1}.\]
Choose  $R_2$ large so that there exists $\theta>0$ such that
\[ \frac{R_2-5}{\sqrt{1-s_\ast}}\geq (1+\theta)R_2.\]
Then \eqref{B2} implies that
\[\text{dist}_{(1+\theta)R}(w,\kappa)+\text{dist}_{(1+\theta)R}\left(\Lambda(w), \frac{2}{p-1}\kappa\right)\leq \frac{\kappa}{p+1}. \qedhere\]
\end{proof}

\section{Proof of Lemma \ref{lem L2 convergence lift to smooth}}\label{sec proof of Lemma 4.6}
	\setcounter{equation}{0}
  
In view of Proposition \ref{epsilonregularity},  we need only to prove the following claim.\\
{\bf Claim.} There exist  two constants $s_\ast, k_{0}$ such that   for any $(x, t)\in B_{1}\times (-2, -1+s_{\ast})$ and $k\geq k_0$,
\begin{equation}\label{lem L2 convergence lift to smooth1}
\bar{E}(s_\ast; x, t, u_{k})<\varepsilon_{0},
\end{equation}
where $\varepsilon_{0}$ is the positive constant in Proposition \ref{epsilonregularity}.

Once this claim is established, an application of Proposition \ref{epsilonregularity} gives the uniform estimate \eqref{5.3}. The $C^{2,1}$ convergence then follows from standard parabolic regularity theory and Arzela-Ascoli theorem.

Similar to the proof of  Lemma \ref{lem small of monotonicity quantity}, we   only   prove the estimate 
 \begin{equation}\label{lem L2 convergence lift to smooth2}
s^{-1}\int_{s}^{2s}\int_{\mathbb{R}^{n}}\tau^{\frac{p+1}{p-1}}|u_{k}(y, t-\tau)|^{p+1}G(y-x,\tau)dyd\tau<\frac{\varepsilon_0}{3},
\end{equation}
whenever $k>k_{0}$ and $0<s\leq s_{\ast}$.

Decompose the above integral into
\[I_{1}=s^{-1}\int_{s}^{2s}\int_{\mathbb{R}^{n}\backslash B_{2}(0)}\tau^{\frac{p+1}{p-1}}|u_{k}(y, t-\tau)|^{p+1}G(y-x,\tau)dyd\tau,\]
and
\[I_{2}=s^{-1}\int_{s}^{2s}\int_{B_{2}(0)}\tau^{\frac{p+1}{p-1}}|u_{k}(y, t-\tau)|^{p+1}G(y-x,\tau)dyd\tau.\]

Concerning $I_{1}$, for any $s\in(0,\frac{1}{8})$, we have
\begin{align}\label{estimate on I1}
I_{1}=&s^{-1}\int_{s}^{2s}\int_{\mathbb{R}^{n}\backslash B_{2}(0)}\tau^{\frac{p+1}{p-1}}|u_{k}(y, t-\tau)|^{p+1}G(y-x,\tau)dyd\tau \nonumber\\
\leq& Cs^{-1+\frac{p+1}{p-1}-\frac{n}{2}}e^{-\frac{1}{s}}\int_{s}^{2s}\int_{\mathbb{R}^{n}\backslash B_{2}(0)}|u_{k}(y, t-\tau)|^{p+1}e^{-\frac{|y-x|^{2}}{2}}dyd\tau\\
\leq &Cs^{-1+\frac{p+1}{p-1}-\frac{n}{2}}e^{-\frac{1}{s}}\int_{t-s-\frac{1}{4}}^{t-s}\int_{\mathbb{R}^{n}\backslash B_{2}(0)}|u_{k}(y, \tau)|^{p+1}e^{-\frac{|y-x|^{2}}{2}}dyd\tau, \nonumber
\end{align}
where $C$ is a positive constant independent of $k$. 

Since $\{u_{k}\}\subset\mathcal{F}_{M}$, we have
\[\begin{aligned}
&\int_{t-s-\frac{1}{4}}^{t-s}\int_{\mathbb{R}^{n}\backslash B_{2}(0)}|u_{k}(y, \tau)|^{p+1}e^{-\frac{|y-x|^{2}}{2}}dyd\tau\\
=&\sum_{j=2}^{\infty}\int_{t-s-\frac{1}{4}}^{t-s}\int_{\{j\leq |y|\leq j+\frac{1}{2}\}}|u_{k}(y, \tau)|^{p+1}e^{-\frac{|y-x|^{2}}{2}}dyd\tau\\
\leq& C\sum_{j=2}^{\infty}(j+1)^{n}e^{-\frac{(j-1)^{2}}{2}}\\
\leq & C
\end{aligned}\]
where $C$ is a positive constant depending on $n $ and $M$. In the process of deriving the above estimate, we have applied the Morrey bound in \eqref{Morreyeatimates} and the fact that we can cover $B_{j+1}(0)\times(t-s-1, t-s)$ by at most $C(n)(j+1)^{n}$ backward parabolic cylinders of size $1$.

Plugging the above  estimate into \eqref{estimate on I1}, we conclude that there exists a positive constant $C$ depending only on $n, M, p$ such that
\[I_{1}\leq Cs^{-1+\frac{p+1}{p-1}-\frac{n}{2}}e^{-\frac{1}{s}}.\]
By taking $s_{\ast, 1}$ small enough, we get for any $0<s<s_{\ast, 1}$ and $(x, t)\in  B_{1}\times (-2, -1+\frac{s}{2})$,
\[I_{1}=s^{-1}\int_{s}^{2s}\int_{\mathbb{R}^{n}\backslash B_{2}(0)}\tau^{\frac{p+1}{p-1}}|u_{k}(y, t-\tau)|^{p+1}G(y-x,\tau)dyd\tau<\frac{\varepsilon_{0}}{6}.\]

For the second term, we have
\[\begin{aligned}
I_{2}\leq & s^{-1}\int_{s}^{2s}\int_{B_{2}(0)}\tau^{\frac{p+1}{p-1}}|u_{\infty}(y, t-\tau)|^{p+1}G(y-x,\tau)dyd\tau\\
+& s^{-1}\int_{s}^{2s}\int_{B_{2}(0)}\tau^{\frac{p+1}{p-1}}\big||u_{k}(y, t-\tau)|^{p+1}-|u_{\infty}(y, t-\tau)|^{p+1}\big|G(y-x,\tau)dyd\tau.
\end{aligned}\]
Since $u_{\infty}$ is smooth in $Q_4^-(0, -1)$, there exists a constant $0<s_{\ast, 2}<\min\{s_{\ast,1},\frac{1}{4}\}$ such  that for any $(x, t)\in  B_{1}\times (-2, -1+\frac{s_{\ast, 2}}{2})$,
\[s_{\ast,2}^{-1}\int_{s_{\ast,2}}^{2s_{\ast,2}}\int_{B_{2}(0)}\tau^{\frac{p+1}{p-1}}|u_{\infty}(y, t-\tau)|^{p+1}G(y-x,\tau)dyd\tau\leq Cs_{\ast,2}^{\frac{p+1}{p-1}}<\frac{\varepsilon_{0}}{12},\]
where $C$ depends only on $n, p$ and $\|u_{\infty}\|_{L^{\infty}(Q_4^-(0, -1))}$.

Once $s_{\ast, 2}$ is fixed, then there exists a positive constant $C$ depending only on $n, p$ and $s_{\ast, 2}$ such that
\begin{align}\label{C4}
I_{2}\leq \frac{\varepsilon_0}{12}+C\int_{-3}^{-\frac{1}{2}}\int_{B_{2}(0)}\big||u_{k}^{p+1}|-|u_{\infty}|^{p+1}\big|dydt.
\end{align}
Since $u_{k}\to u_{\infty}$ in $L^{p+1}_{loc}$,   we can choose $k_{0}$ so that for any   $k>k_{0}$,
\[C\int_{-3}^{-\frac{1}{2}}\int_{B_{2}(0)}\big||u_{k}^{p+1}|-|u_{\infty}|^{p+1}\big|dydt<\frac{\varepsilon_0}{12}\]
and then by \eqref{C4},
\[I_2\leq \frac{\varepsilon_0}{6}.\]
The proof is complete by taking $s_{\ast}=\min\{s_{\ast, 1}, s_{\ast, 2}, \frac{1}{8}\}$.

\end{document}